\documentclass[reqno,10pt]{amsart}
\usepackage[english]{babel}
\usepackage{amsthm}
\usepackage{amsmath}
\usepackage{amssymb}
\usepackage[latin1]{inputenc}

\usepackage[all]{xy}
\usepackage{dsfont}
\usepackage[bookmarksnumbered,colorlinks]{hyperref}
\usepackage{graphics}
\usepackage{enumerate}
\def\bb#1\eb{\textcolor{blue}
{#1}} %
\def\br#1\er{\textcolor{red}
{#1}} %
\newcommand{\df}{{\rm d}}
\newcommand{\dist}{{\rm d}}
\newcommand{\R}{\mathds R}
\newcommand{\N}{\mathds N}
\newcommand{\Z}{\mathds Z}

\newcommand{\be}{\begin{equation}}
\newcommand{\ee}{\end{equation}}

\def\ri#1{\mathring{#1}}

   \def\br#1\er{\textcolor{red}{#1}} %
      \def\bb#1\eb{\textcolor{blue}{#1}} %

\title[On Finsler metrics]{On the definition and examples\\ of Finsler metrics}

\author[M. A. Javaloyes]{Miguel Angel Javaloyes}
\address{Departamento de Matem\'aticas, \hfill\break\indent
Universidad de Murcia, \hfill\break\indent
Campus de Espinardo,\hfill\break\indent
30100 Espinardo, Murcia, Spain}
\email{majava@um.es}

\author[M. S\'anchez]{Miguel S\'anchez}
\address{Departamento de Geometr\'{\i}a y Topolog\'{\i}a, Facultad de Ciencias, \hfill\break\indent
 Universidad de Granada,\hfill\break\indent
 Campus Fuentenueva s/n,
 \hfill\break\indent 18071 Granada, Spain}
\email{sanchezm@ugr.es}

\date{21.11.2011}


\begin{document}
\newtheorem{thm}{Theorem}[section]
\newtheorem{prop}[thm]{Proposition}
\newtheorem{lemma}[thm]{Lemma}
\newtheorem{cor}[thm]{Corollary}
\theoremstyle{definition}
\newtheorem{defi}[thm]{Definition}
\newtheorem{notation}[thm]{Notation}
\newtheorem{exe}[thm]{Example}
\newtheorem{conj}[thm]{Conjecture}
\newtheorem{prob}[thm]{Problem}
\newtheorem{rem}[thm]{Remark}
\newtheorem{conv}[thm]{Convention}
\newtheorem{crit}[thm]{Criterion}

\begin{abstract}
For a standard Finsler metric $F$ on a manifold $M$, its domain is
the whole tangent bundle $TM$ and its fundamental tensor $g$ is
positive-definite. However, in many cases (for example, the
well-known Kropina and Matsumoto metrics), these two conditions
are relaxed, obtaining then either a {\em pseudo-Finsler metric}
(with arbitrary $g$) or a {\em conic Finsler metric} (with domain
a ``conic'' open domain of $TM$).

Our aim is twofold. First, to give an account of quite a few
subtleties which appear under such generalizations, say, for {\em
conic pseudo-Finsler metrics} (including, as a previous step, the
case of {\em Minkowski conic pseudo-norms} on an affine space).
 Second, to provide some criteria  which determine when a
pseudo-Finsler metric $F$ obtained as a general homogeneous
combination of Finsler metrics and one-forms is again a Finsler
metric ---or, with more accuracy, the conic domain where $g$
remains positive definite. Such a combination generalizes the
known $(\alpha,\beta)$-metrics in different directions.
Remarkably, classical examples of
Finsler metrics 
are reobtained and extended, with explicit computations of their
fundamental tensors.
\end{abstract}

\maketitle

\newpage
\tableofcontents
\newpage

\section{Introduction}

Finsler Geometry is a classical branch of Differential Geometry,
with a well-established background. However, its applications to
Physics and other natural sciences (see, for example, \cite{AIM93,
Asa85, Mat89b, Ran41, SSI08}) introduced some ambiguity in what is
understood as a Finsler manifold. The  standard  definition of a
Finsler metric $F$ on a manifold $M$ entails that $F$ is defined
{\em on the whole tangent bundle} $TM$ and that {\em strong
convexity} is satisfied, i.e., its fundamental tensor $g$ is
positive definite \cite{AP,Akbar-Zadeh06,BaChSh00,ChSh05,Mo06,shen2001,Sh01}.
In fact, the non-degeneracy of $g$ is essential for many purposes
which concern connections and curvature, and its
positive-definiteness implies the fundamental inequality, with
important consequences on the associated Finsler distance,
minimizing properties of geodesics, etc.

However, in many interesting cases, the metric $F$ is defined only
in some {\em conic domain} $A\varsubsetneq TM$; recall, for
example, Matsumoto \cite{Mat89b} or Kropina metrics
\cite{Kro59,Mat72,Shi78},  or some metrics constructed as in the
Randers case \cite[Chapter 11]{BaChSh00} (see also
Corollary~\ref{randers} below). Typically, this happens when $F$
is not positive-definite or  smoothly inextendible in some
directions and, so, these directions are removed from the domain
of $F$. Many references consider as the essential ingredient of a
Finsler metric $F$ that it behaves as a (positively homogeneous)
pointwise norm in some conic open subset $A\subset TM$
\cite{AIM93,Asa85,BeFa2000,Br,Mat86}. In
this case, the role of the positive-definiteness of $g$ may remain
somewhat vague. In principle, one admits the {\em pseudo-Finsler}
case when $g$ is allowed to be non-positive definite. If the
non-degeneracy of $g$ is required, then one can redefine $A$ so
that degenerate directions are also suppressed. If, moreover,
 positive-definiteness is required, $A$ will contain only those
directions where this property holds. However, notice that it is
important then to have criteria which determine exactly the domain
$A$, as well as the expected properties for $F$ and $g$ in each
case.

 The first aim of this work is to  study these cases
by making two basic extensions of the  standard  notion of Finsler
metric: {\em pseudo-Finsler metrics} (if one admits that its
fundamental tensor $g$ can be non-positive definite) and {\em
conic Finsler metrics} (if one admits that the domain of $F$ is not
necessarily the whole $TM$ but an open conic domain); if both
extensions are done at the same time we speak of a {\em conic
pseudo-Finsler metric}.  Once the general properties of
these extensions are studied, our next aim is to determine the
regions where strong convexity holds for natural classes of conic
pseudo-Finsler metric. Recall that
 one of the distinctive aspects of Finsler geometry, in
 comparison with
the Riemannian one, is how Finsler metrics (and one-forms) can be
combined to obtain new Finsler metrics.  In Riemannian
geometry the natural operations are just the addition and
conformal changes. But in Finsler geometry there is a big amount
of possibilities providing endless families of Finsler metrics
(see Section \ref{ls3}).  As a remarkable goal, we compute
the fundamental tensor $g$ of these combinations explicitly, in
such a way that the domain where $g$ is positive definite  becomes
apparent.  The paper is divided then in three sections.

In Section \ref{s1}, our study starts at the very beginning, by
discussing accordingly the extensions of  Minkowski norms into
{\em Minkowski pseudo-norms} and {\em Minkowski conic norms}
---or, in general, {\em Minkowski conic pseudo-norms}. After a
preliminary review of properties of Minkowski norms in the first
subsection, in the second one we focus in their generalizations.
As the weakening of these properties may yield the loss of the
triangle inequality, we study at the same time norms and norms
which satisfy the strict triangle inequality (including a
discussion about the weakening of differentiability in
Proposition \ref{pcontinuousnorms}, Remark
\ref{rcontinuousnorms}). In the third subsection we show how all
these norm-related notions are characterized by looking at the
corresponding unit  (conic) sphere or indicatrix $S$, which allows one to
reconstruct the norm from  suitable candidates to unit
(conic) ball $B$
 (Proposition \ref{pfunda}, Theorem \ref{tfunda}).  The interplay between the
 convexities of $S$, $B$ and the conic domain $A$ of the Minkowski conic pseudo-norm is stressed.
 Finally, in the fourth subsection, simple examples
 (constructed on $\R^2$ from a curve which  plays the role of $S$)
show the optimality of the results.
 We stress in both, results and examples, that even though the triangle
inequality may not hold for Minkowski pseudo-norms, their forward (resp. backward)
affine balls still constitute a basis of the topology, as a
difference with the conic case (Proposition \ref{ppseudi}).
However, for conic Minkowski norms, the sets of all the forward
and backward affine balls generate the topology as a subbasis
(Proposition \ref{opennormballs}), and suitable triangle
inequalities occur when the conic domain $A$ is convex.

In Section \ref{s2}, the general properties of a conic
pseudo-Finsler metric $F$, as well as those for the particular
pseudo-Finsler and conic Finsler cases, are studied. It is divided
into five subsections with, respectively, the following aims: (1)
A discussion on the general definition, including the subtleties
inherent to our general choice of conic domain $A\subset TM$ and
the possibility of extending it to all $TM$. (2) To introduce the
notion of {\em admissible curve} (with velocity in $A$) and,
associated to it: the conic partial ordering $\prec$, the {\em
Finsler separation $d_F$} (which extends the classical Finsler
distance), and the open  forward and backward balls, which
are shown to be always topologically open subsets but, in general,
they do not constitute a basis for the topology (in fact, $d_F$
may be identically zero in the conic case). So, the condition of
being Riemannianly lower bounded for $F$ is introduced to guarantee
that the forward (resp. backward) balls give a basis of the topology. (3) A
discussion on the role of the open balls in this general
framework. In particular, we show that, in the pseudo-Finsler
case, the open forward balls still constitute a basis of the
topology, and the Finsler separation $d_F$ is still a generalized
distance. However, we stress that the corresponding
$d_F$-balls (obtained as a length space) may look very different
from the Finsler case
---in particular, they differ from the affine balls for Minkowski
pseudo-norms whenever its fundamental tensor $g$ is not
positive semi-definite. (4) To introduce geodesics as critical
points of the energy functional for admissible curves joining two
fixed points. In the non-degenerate directions they are related to
the Chern connection, and characterized by means of the geodesic
equation (univocally determined from their velocity at some
point); otherwise, their possible lack of uniqueness becomes
apparent (see Example \ref{ex2a}). Their
minimization properties only appear in the conic Finsler case, and
some subtleties about them are especially discussed. (5) A final
summary is provided, making a comparison between the Finsler,
pseudo-Finsler and conic Finsler cases at three levels: (i)
Finsler separation $d_F$, (ii) geodesics, and (iii) the particular
affine Minkowski case.

In Section \ref{ls3},
 we study a general homogeneous functional combination of
$n$ (conic) Finsler metrics $F_1, \dots , F_n$ and $m$ one-forms
$\beta_{n+1}, \dots , \beta_{n+m}$. Such a $(F_1, \dots , F_n,$
$\beta_{n+1}, \dots , \beta_{n+m})$-metrics constitute an obvious
generalization of the known $(\alpha, \beta)$-metrics (see
\cite{Mat72} and also
\cite{KAM95,MSS10,Mat89,Mat91b,Mat91c,Mat92,SSI08}) and
$\beta$-deformations \cite{Rom06}. In the first subsection, we
compute explicitly its fundamental tensor $g$ and derive general
conditions to ensure that $g$ is positive definite in Theorem
\ref{central}. Then, some simple particular cases (Corollaries
\ref{sum} and \ref{rrandersr}) and consequences (Remark
\ref{consequences} and Corollary \ref{cisom}) are stressed. In the
second subsection we focus on the simple case $n=1=m$. The metric
$F$ obtained in this case coincides with the so-called
$\beta$-deformation of a preexisting Finsler metric $F_0$, and it
is called here a {\em $(F_0,\beta)$-metric}.  In particular,
classical Matsumoto, Randers and Kropina metrics are reobtained
and extended, including the explicit computation of their
fundamental tensor $g$, as well as the conic domain for its
positive-definiteness. As an example of the possibilities of our
approach, we conclude with a further extension to metrics
constructed from a pair $(F_1,F_2)$ of Finsler ones. Finally, as
an Appendix, in the third subsection,  our computations are
compared with those by Chern and Shen \cite{ChSh05} for
$(\alpha,\beta)$-metrics.

Finally, we would like to emphasize that, even though our study is
quite extensive, it is not by any means exhaustive. There are
still natural questions related to the conic pseudo-Finsler
metrics which have not been well studied yet and, as some of the
ones studied in this paper, may have a non-trivial and subtle
answer. Due to the increasing activity in this field, we encourage
 the readers to make further developments.

\section{Generalizing Minkowski norms}\label{s1}

\subsection{Classical notions} First, let us recall the following classical concepts.

\begin{defi}\label{dnorms} Let $V$ be a real
vector space  of finite dimension $N\in \N$ and consider a
function $\|\cdot\|:V\rightarrow\R$ being

\begin{enumerate}[(i)]

\item positive: $\| v \| \geq 0$, with equality  if  only
if $v=0$,

\item positively homogeneous: $\| \lambda v \| =
\lambda \|v\|$ for all  $\lambda > 0$.
\end{enumerate}
\smallskip

\noindent Then, $\| \cdot \|$ is called:
\begin{enumerate}[(a)]
\item A {\em positively homogeneous norm}, or simply {\em
norm},  if it satisfies the triangle inequality, i.e.: $\| v+w
\| \leq \| v \| + \| w \|$.

\item A {\em norm with strict triangle inequality}, if the equality
in the triangle inequality is satisfied only if $v=\lambda w$ or
$w=\lambda v$ for some $\lambda\geq 0$.

\item A {\em Minkowski norm}, if: (c1) $\| \cdot \|$ is
smooth\footnote{ For simplicity, ``smooth'' will mean
$C^\infty$ even though differentiability $C^2$ will be enough for
most  purposes. } away from $0$, so that the {\em fundamental
tensor field} $g$ of $\| \cdot \|$ on $V\setminus\{0\}$ can be
defined as the Hessian of $\frac 12 \| \cdot \| ^2$,  and (c2) $g$
is pointwise positive-definite.
\end{enumerate}
\end{defi}
About the last definition, recall  that the Hessian  of any smooth
function $f$ is written as Hess$f(X,Y)=X(Y(f))-\nabla^0_XY(f)$ for
any vector fields $X,Y$, where $\nabla^0$ is the natural affine
connection of $V$. Moreover,
Minkowski norms always satisfy the strict
triangle inequality (see for example \cite[Theorem 1.2.2]{BaChSh00}).

Let us collect some properties of the fundamental tensor $g$ that
come directly from its definition and the positive homogeneity.
\begin{prop}\label{fundamentalprop}
Given a Minkowski norm $\|\cdot\|$ and $v\in V\setminus\{0\}$, the
fundamental tensor $g_v$ is given as:

\begin{equation}\label{eg} g_v(u,w):=
\frac{\partial^2}{\partial t\partial s}
G(v+tu+sw)|_{t=s=0},\end{equation} where $u,w\in V$ and
$G=\frac 12\|\cdot\|^2$. Moreover, $v\mapsto g_v$ is positively
homogeneous of degree 0 (that is, $g_{\lambda v}=g_v$ for $\lambda>0$) and it
satisfies
\begin{align}\label{propfundtensor}
g_v(v,v)&=\|v\|^2,& g_v(v,w)&= \frac{\partial}{\partial
s}G\left(v+s w\right)|_{s=0}.
\end{align}
Moreover, $v$ is $g_v$-orthogonal to the unit sphere $S$, and the metric
$g$ is positive definite if and only
if so is its restriction to $S$.
\end{prop}
Any norm can be characterized by its closed unit ball $B=\{x\in
V:\|x\|\leq 1\}$ and its unit sphere or {\em indicatrix} $S=\{x\in
V:\|x\|=1\}$.  For the next result, recall that $V$ is endowed
with a natural affine connection $\nabla^0$ (as so is any vector
 or affine space) and, then, any smooth hypersurface admits a
 second fundamental form $\sigma^\xi$ for each
  {\em transverse}\footnote{ Even though a Minkowski norm would allow  one
to define locally a normal direction to the hypersurface (recall
for example \cite[Def. 3.2.2, Cor. 3.2.3 ]{Th}), this will not
play any role here.}  vector field  $\xi$ (see
\eqref{nablacanonica} below).

\begin{prop}\label{pnormball} Let $\| \cdot \|:V\rightarrow \R$ be a
norm and $B$ its closed unit ball. Then
\begin{enumerate}[(i)]
\item Vector $0$ belongs to the interior of $B$, $0\in  \ri B$.

\item $B$ is compact.

\item $B$ is convex (i.e., $u,v\in B$ implies $ \lambda u +
(1-\lambda)v\in B$, for all $\lambda\in [0,1]$).
\end{enumerate}
 Moreover,
\begin{enumerate}[(i)]
\item[(iv)] the norm $\| \cdot \|$ satisfies the strict triangle
inequality iff $B$ is strictly convex (i.e. $u,v\in B$, with
$u\neq v$, implies $ \lambda u + (1-\lambda)v\in  \ri B$, for all
$\lambda\in (0,1)$).

\item[(v)] the norm $\| \cdot \|$ is a Minkowski one iff
$S=\partial B$ is a smooth  hypersurface embedded in $V$ and it is
strongly convex (i.e. the second fundamental form $\sigma^\xi$ of
$S$ with respect to some, and then any, transverse vector $\xi$
pointing out to $\ri  B$ is positive definite).
\end{enumerate}
\end{prop}
\begin{proof}
Standard arguments as in \cite{BaChSh00, Th} are used. Namely,
assertions in $(i)$ and $(ii)$ follow from the fact that all the
norms on $V$ are equivalent \cite[p. 29]{Th}, and $(iii)$, as well
as ($(iv)\Rightarrow$), is straightforward from the triangle
inequality applied to $ \lambda u + (1-\lambda)v$. For
($(iv)\Leftarrow$), in the non-trivial case when $\{u, v\}$ is
linearly independent,  put $\tilde{u}=u/\| u \|, \tilde{v}=v/\| v
\|\in \partial B$ so that
$$
z=\frac{\| u \|}{\| u \|+ \| v
\|}\tilde{u} + \frac{\| v \|}{\| u \|+ \| v
\|}\tilde{v} \in \ri B.
$$
Then, $1>\| z \| =\| u +v\|/ (\| u \|+ \| v \|)$, as required.

For ($(v)\Rightarrow$), recall that positive homogeneity implies
that $1$ must be a regular value of $\| \cdot \|$ and, so, $S$ is
a closed smooth hypersurface in $V$. Moreover, the opposite $\xi$
of the vector position is transverse to $S$ and points out to $\ri
B$. So, for any vector fields $X, Y$  tangent to $S$, the
decomposition of the canonical connection $\nabla^0$ on $V$,
\begin{equation}\label{nablacanonica}
 \nabla^0_XY =\nabla^\xi_XY + \sigma^\xi(X,Y)\xi,
\end{equation}
holds under standard conventions. Putting again $G=\frac 12\|\cdot\|^2$,
then \begin{equation}\label{esconv0} g(X,Y)={\rm Hess}\, G(X,Y)=
-\nabla^0_XY(G),\end{equation} and using
(\ref{nablacanonica}),
\begin{equation}\label{esconv}
g(X,X) =-\nabla^\xi_XX(G) - \sigma^\xi(X,X)\xi(G)=-
\sigma^\xi(X,X)\xi(G).
\end{equation}
As positive homogeneity implies that $-\xi(G)>0$, we have that $g(X,X)>0$ iff $\sigma^\xi(X,X)>0$. For ($(v)$ $\Leftarrow$), by last statement in Proposition \ref{fundamentalprop}, we only have to prove that the restriction of $g$ to
the indicatrix $S$ is positive definite. Moreover, notice that the
positive homogeneity implies that $-\xi$ is transverse and
$\|\cdot\|$ is smooth away from $0$.  So, \eqref{esconv} can be applied again and we are done.
\end{proof}
\subsection{Generalized notions}
Next, let us generalize the notion of Minkowski norm.

\begin{defi}\label{dnormsextended}

A {\em Minkowski pseudo-norm} on $V$ is a map $\| \cdot
\|:V\rightarrow \R$ which satisfies $(i)$,
$(ii)$ and $(c1)$ in Def. \ref{dnorms} (i.e., $g$ is not
necessarily positive definite).

 A  {\em Minkowski conic norm} on $V$ is a map  $\| \cdot
\|:A\rightarrow\R$, where $A\subset V$ is
 a {\em conic
domain}   (i.e., $A$ is open, non-empty and satisfies that if
$v\in A$, then $\lambda v\in A$ for all $\lambda>0$), which
satisfies the conditions $(i)$, $(ii)$, $(c1)$ and $(c2)$ in Def.
\ref{dnorms} for all  $v\in A$  
(i.e. $g$ is positive definite, but the conic domain $A$ may
not be all $V$ and, in this case, it excludes vector $0$).

A {\em Minkowski conic pseudo-norm} on $V$ is a map $\| \cdot
\|:A\rightarrow\R$, which satisfies $(i)$, $(ii)$ and (c1) in Def.
\ref{dnorms}  for all
 $v\in A$,     where $A\subset V$ is a  conic domain 
(i.e., the two
previous extensions of the notion of Minkowski norm are allowed
simultaneously).
\end{defi}

 For simplicity, we will assume typically that $A$ is
connected. As in the case of norms, we do not assume that
$\|\cdot\|$
 is reversible, and 
 we can define two types of  {\em affine (normed) balls}  depending on the order
 we compute the substraction. Namely, for any Minkowski conic pseudo-norm the {\em forward and backward
 affine  open balls} of center $v$
 and radius  $r> 0$  are defined, respectively, as
 \begin{align}\label{eopenballs}
 B^+_v(r)&=\{x\in v+A : \|x-v\|<r\}& \text{and} &&B^-_v(r)=\{x\in v-A : \|v-x\|<r\}.
  \end{align}
  where $v\pm A:=\{v\pm w: w\in A\}$.
 In the case that $g$ is not positive-definite, the behavior of
these balls may differ dramatically from the behavior of the
metric balls obtained from a length space (see  Example \ref{ex4}
and Remark \ref{pseudoaffineballs} below), even though the
continuity of $\|\cdot\|$ allows one to ensure that they are open
subsets. For closed balls $\bar B^\pm_v(r)$, the non-strict
inequality $\leq$ is used instead of $<$ in \eqref{eopenballs};
recall that these balls are closed in $v\pm A$ but not in $V$
(except if $V=A$). For the forward and backward spheres
$S_v^\pm(r)$, equalities replace the inequalities in
\eqref{eopenballs},
and the map
\begin{equation}\label{espheres}
\varphi:S^+_v(r)\rightarrow S^-_v(r) \quad \quad w\mapsto 2v-w
\end{equation}
is a homeomorphism. As in the case of norms, we will work by
simplicity with $B=\bar B_0^+(1)$ and $S=S_0^+(1)$ (which is the boundary of
$B$ in $A$).
  Recall also that $0\not\in A$ (and thus $0\not\in B$) except in the
case that $A=V$, i. e., when $\parallel \cdot \parallel$ is a
pseudo-norm.

\begin{rem}\label{r26}
 Observe that Proposition \ref{fundamentalprop} is
extended directly to this case due to its local nature and the
positive homogeneity. Moreover, the expressions
\eqref{nablacanonica} and \eqref{esconv} are also directly
transplantable to Minkowski conic pseudo-norms.
\end{rem}

 The following technical properties must be taken into
account.

\begin{prop}\label{plema}
 Let $\parallel \cdot \parallel: A\rightarrow \R$ be a Minkowski conic
pseudo-norm. Then
\begin{enumerate}[(i)]
\item the indicatrix $S$ is a hypersurface embedded in $A$ as a
closed subset, and the position vector at each point is transverse
to $S$,

\item if $\parallel \cdot \parallel$ is a pseudo-norm ($A=V$), $S$
is diffeomorphic to a  sphere.
\end{enumerate}

\end{prop}
\begin{proof} For $(i)$, 
positive homogeneity implies that the differential of $\| \cdot
\|$ does not vanish on the position vector and, so, $S$ is the
inverse image  of the regular value $1$.

For $(ii)$, consider the norm $\| \cdot \|^E$ associated to any
auxiliary Euclidean product on $V$, and let $S^E$ be its unit
sphere, and observe that the map $S^E\rightarrow S, v\mapsto v/\|
v \|$ and its inverse are smooth (as so is $\| \cdot \|$).
\end{proof}

\begin{rem}
 As in the proof of the previous proposition, when necessary, we
will consider an auxiliary  norm $\| \cdot \|^E$ associated to
some Euclidean product $g^E$ on $V$, and the results obtained will be independent
of the choice of $g^E$. In particular, $V$ can be regarded as a
Riemannian manifold with Levi-Civita connection $\nabla^0$, the
spheres $S_v^\pm(r)$ as Riemannian submanifolds with the metric
induced by $g^E$, and the map \eqref{espheres} as an
isometry. The  $\|\cdot \|^E$-open balls will be
denoted with a  superscript $E$, say as in
$B^E_v(r)$.
\end{rem}

Now, let us focus on Minkowski conic norms. First, we will show
that the forward  and backward open balls constitute a subbasis
for the topology  (to check  the optimality of this result,
see Examples \ref{ex2a} and \ref{ex3}, and Proposition \ref{ppseudi}
below).

 \begin{prop}\label{opennormballs}
Let $\|\cdot\|:A\rightarrow \R$ be a conic Minkowski norm with $A$ connected. Then,  the collection of subsets
\[\{B_{v_1}^+(r_1)\cap B_{v_2}^-(r_2):\text{$v_1,v_2\in V,$  $r_1,r_2>0$}\}\] is a topological basis of $V$.
\end{prop}
 \begin{proof}  By using translations, it is enough to show that
they yield a topological basis around $0 \in V$.

 Choose any $v_0\in S$ and an auxiliary Euclidean $g^E$ such that
 $v_0$ is also unit and orthogonal to $S$ for $g^E$. Let $\theta_{v_0}(v)$ be the $g^E$-angle
 between $v\in V\setminus \{0\}$ and $v_0$.
 Fix $\theta_0\in (0,\pi/2)$
satisfying  both, every $v\in V$ with $\theta_{v_0}(v)<\theta_0$
lies in $A$, and for all such $v$ in $S$, the minimum of the
$g^E$-principal curvatures of $S$ at $v$ (in the inward direction
to $B$) is greater than some constant $k>0$. Such a constant $k$
exists from the positive definiteness of $g$, putting $\xi$ in
\eqref{esconv} as the inward $g^E$-unit normal (in particular,
$\xi_{v_0}=-v_0$). Then, the $g^E$-round sphere through $v_0$ with
radius $r_k=1/k^2$ and center $x^+$ in the $\xi_{v_0}$ direction
remains outside of $B \cap \{v\in V: \theta_{v_0}(v)<\theta_0\}$.
Observe that, by the isometry $\varphi$ in \eqref{espheres}, the
analogous properties hold for $-v_0$ and $S^-_0(1) (=\varphi(S))$
with the same value of $k$ (and thus  $r_k$), obtaining then a
second center $x^-$.

Now, for any neighborhood $U$ of 0,  there exists some
$\epsilon>0$ such that $x^+_\varepsilon=x^+-(1-\varepsilon) v_0$
and $x^-_\varepsilon=x^-+(1-\varepsilon) v_0$ satisfy
$$B^E_{x^+_\varepsilon}(r_k) \cap B^E_{x^-_\varepsilon}(r_k)\subset
U,$$ and the required property follows because, from the
construction, $$0\in B^+_{(\epsilon-1)v_0}(1) \cap
B^-_{(1-\epsilon)v_0}(1) \subset B^E_{x^+_\varepsilon}(r_k) \cap
B^E_{x^-_\varepsilon}(r_k).$$
\end{proof}

\begin{rem}\label{fundineqRem}
We  emphasize that in the conic Minkowski case the strict triangle
inequality still holds for every $v_1,v_2\in A$ such that $t
v_1+(1-t)v_2\in A$  for $t\in(0,1)$
 (see the proof of parts
 $(iv)$ and (v) in Proposition \ref{pnormball}, or Theorem \ref{tfunda} below); in particular, it holds for any $v_1, v_2\in A$ if $A$ is
 convex.
  Analogously, if additionally $v_1\neq 0$ then
$\|\cdot\|$ satisfies the {\it Fundamental Inequality}:
\[\frac{\partial}{\partial t} \big(\|v_1+t v_2\|\big)|_{t=0}\leq \|v_2\|,\]
  or equivalently,
\begin{equation}\label{fundineq}
g_{v_1}(v_1,v_2)\leq \|v_1\| \|v_2\|,
\end{equation}
(use Eq. \eqref{propfundtensor}),  where  the equality holds
if and only
 if $v_2=\lambda v_1$ for some $\lambda\geq 0$ (see \cite[Section 1.1]{ChSh05}).

With more generality,  the  (strict) triangle inequality for
$v_1,v_2\in A$ in a Minkowski conic pseudo-norm holds when:
\begin{enumerate}[(i)]
\item $t v_1+(1-t)v_2\in A$, \item $g$ is positive definite in the
direction ${t v_1+(1-t)v_2}$,
\end{enumerate}
for every $t\in(0,1)$.  As, essentially, the triangle inequality implies the
fundamental one (see \cite[p. 9]{BaChSh00}), conditions (i) and
(ii) are also sufficient to ensure this inequality on $v_1$ and
$v_2$. In particular, if $g_{v_1}$ is positive-definite then both,
the triangle inequality and the fundamental one, hold in a
neighborhood of $v_1$. Moreover,  these inequalities also hold for
any conic Minkowski pseudonorm under more general hypotheses
related to the possible convexity of the indicatrix somewhere. For
example, fixing $v_1\in A$ the fundamental inequality holds for
any $v_2\in A$ if the hyperplane $H$ tangent to the indicatrix $S$
at  $v_1/\| v_1
\|$  touches the
closed unit ball only at $v_1$ (this implies that $g_{v_1}$ is
positive semi-definite and, even when non-positive definite, the
 directions $g_{v_1}$-orthogonal to $v_1$ are those tangent to $H$). Indeed,
given $v_1,v_2$, we can assume that $v_2=\lambda v_1+w$, with
$\lambda\geq 0$ and $g_{v_1}(v_1,w)=0$ (if $\lambda<0$, the
fundamental inequality holds trivially). Observe that the
hypothesis on the hyperplane (transplanted to the parallel
hyperplane at $\lambda v_1$) yields $\lambda\| v_1\|\leq \|v_2\|$
with equality only when $\lambda v_1 =v_2$. Then
\[g_{v_1}(v_1,v_2)=\lambda \|v_1\|^2\leq \|v_1\|\|v_2\|,\]
and the required fundamental inequality follows. 
\end{rem}

The word ``Minkowski'' in Definition \ref{dnormsextended}
comprises two properties for the defined objects: (i) they are
smooth away from 0, and (ii) in the case of (conic) norms, the
fundamental tensor $g$ is positive definite. Recall, that for a
classical norm as in part $(a)$ of Definition \ref{dnorms}, one has only the
weaker properties of continuity and triangle inequality, the
former deduced from the latter. This can be extended to the conic
case but, as the triangle inequality is involved, the previous
remark suggests to impose convexity for the domain $A$
---in particular, $A$ will be connected. Recall that, under this assumption, if
there exists a vector $v\in V$ such that $v,-v\in A$ then $0\in A$ and $A=V$.

\begin{defi}\label{dcontinuousnorms} Let $A$ be a convex conic open subset of $V$. We say that a map $\| \cdot
\|:A\rightarrow\R$ satisfying $(i)$ and $(ii)$ in Def.
\ref{dnorms}  for all $v\in A$ is:

\begin{enumerate}[(i)]
\item a  {\em conic norm} on $V$ if it satisfies the triangle
inequality ($(a)$ in Def. \ref{dnorms}), \item a {\em conic norm
with strict triangle inequality} if it satisfies
 $(b)$ in Def. \ref{dnorms}.
\end{enumerate}
\end{defi}
\begin{prop}\label{pcontinuousnorms}
Any conic norm $\| \cdot \|:A\rightarrow\R$ is continuous, its
open forward and backward balls are open subsets of $V$, and its
indicatrix $S$ is a topological hypersurface, which is closed as
a subset of $A$ and homeomorphic to an open subset of the usual
sphere.
\end{prop}
\begin{proof}
Let us show first that  the forward and the backward affine balls
are open. Given $x\in B^+_v(r)$, as $A$ is open, we can fix a
basis $e_1,e_2,\ldots,e_N$ of $V$ contained in $A$, and such
that $z:=x-v=\sum_{i=1}^N z^ie_i$ with $z^i>0$ for all $i=1,\ldots,N$.
Denote $C=\max \{\|e_1\|,\|e_2\|,\ldots, \|e_N\| \}$ and $y=\sum_{i=1}^N
y^i e_i$ with $y_i>0$ for all $i$. 
By the
triangle inequality, if $0<\lambda <1$, the open subset
\[O_\lambda=\{v+\lambda (x-v)+y: C \sum_{i=i}^N y^i <r-\lambda \|x-v\|\}\]
 is contained in $B^+_v(r)$. Moreover, $x\in O_\lambda$ when
$(1-\lambda)z$ can be chosen as one such $y$, i.e., whenever
 \[(1-\lambda)C \sum_{i=1}^N z^i <r-\lambda\|x-v\|.\]
This holds for  $\lambda$ close to one, as required. The backward
case is analogous.

For the continuity, given $x_0\in A$ and $0<\varepsilon<1$, let us
find an open neighborhood $\Omega\subset A$ of $x_0$ such that
$|\|x\|-\|x_0\||<\varepsilon$ for all $x\in \Omega$. Choose $y_\pm
= (1\pm \delta)x_0$ with $\delta=\frac{\varepsilon}{2\|x_0\|}$
and put $\Omega=B^+_{y_-}(\varepsilon)\cap
B^-_{y_+}(\varepsilon)$. This subset is open by the first part and,
if
$x\in \Omega$: 
\[\|x\|\leq \|x-y_-\|+\|y_-\|< \varepsilon+\|x_0\|, \quad
\|x_0\|\leq \|y_+\|\leq \|y_+-x\|+\|x\|< \varepsilon+\|x\|\] as
required.

For the last assertion, consider an auxiliary Euclidean norm
$\|\cdot\|^E$ with unit sphere $S^E$ and put $S_A^E:= S^E \cap A$.
Then, the map
$$
\rho: S^E_A \rightarrow S, \quad \quad v\mapsto v/\| v \|,
$$
is a homeomorphism, and the required properties of $S$ follow.
\end{proof}
\begin{rem}\label{rcontinuousnorms}
 In the case of  (conic) pseudo-norms  (i.e., only $(i)$ and $(ii)$ in Def.
\ref{dnorms} are fulfilled) the continuity does not follow
because there is no triangle inequality. So, $S$ may be not closed
in $A$ (nor a topological hypersurface) and its affine open balls
may be non-open as subsets of $V$. Throughout this paper all the
conic pseudo-norms will be Minkowski, that is, we will assume that
they are smooth away from $0$. Nevertheless, we will also discuss
next those which satisfy the (strict or not) triangle inequality,
even when $g$ is only positive semi-definite and, so, they are not
Minkowski conic norms.
\end{rem}

\subsection{Characterization through the unit ball}

In analogy to Proposition \ref{pnormball}, the unit balls can be
described as follows.

\begin{prop} \label{pfunda} Let $\parallel \cdot \parallel: A\rightarrow \R$ be a Minkowski conic
pseudo-norm. Then
\begin{enumerate}
\item[(a1)]  $B$  is a closed subset of $A$ which intersects all the
directions  $D_v:=\{\lambda v: \lambda >0$\}, $v\in A$.

\item[(a2)] $B$ is starshaped from the origin, i.e., $v\in B$ implies
$\lambda v\in B$ for all $\lambda\in(0,1)$.

\item[(a3)] The boundary $S$ of $B$ in $A$ is a smooth
hypersurface and a closed subset of $A$ such that the position
vector at each $v\in S$ is transversal (not tangent to
$S$).

\item[(a4)] For each $v\in B\setminus\{0\}$ there exists a
(necessarily unique)\footnote{It is obvious that the
uniqueness of $\lambda$ follows from the definition of Minkowski
conic pseudo-norm. However, we point out here that it also follows
from the previous three items, to stress the independence of the
hypotheses in the next theorem.} $\lambda>0$ such that
$v/\lambda \in S$.
\end{enumerate}
\end{prop}
\begin{proof}
All the properties are straightforward from the definition of $B$
and $S$ (for $\it (a3)$, use part $(i)$ of Proposition
\ref{plema}).
\end{proof}
Conversely, the unit balls characterize the different types of
conic pseudo-norms.  In fact, the following theorem (which also
strengthens Proposition \ref{pnormball}) gives a very intuitive
picture of all the types of Minkowski conic pseudo-norms defined
above. Recall that the classic notions of  convexity for $B$
(i.e., as  a neighborhood) were included in Proposition
\ref{pnormball}, and the notions of convexity for $S$ (its
boundary hypersurface) are included in the next theorem.

\begin{thm} \label{tfunda} Let $A$ be a conic domain of $V$ and $B$ a
subset of $A$ which satisfies all the properties $\it (a1)$ to $\it (a4)$ in
Proposition \ref{pfunda}. Then, the map
$$
\| \cdot \|_B: A\rightarrow \R , \quad \quad
v\mapsto \mbox{{\rm Inf}}\, \{\lambda >0: v/\lambda\in B\},
$$
is a Minkowski conic pseudo-norm and its closed unit ball is equal
to $B$. Moreover,
\begin{enumerate}
\item[(i)] $\| \cdot \|_B$ is a Minkowski pseudo-norm iff $S$ is
homeomorphic to a sphere. In this case,   $0\in \ri B$ and
$\|\cdot\|_B$ is continuous in $0$.
\item[(ii)] $\| \cdot \|_B$
is a conic Minkowski norm iff $S$ is strongly convex.
\end{enumerate}
 Assume now that $A$ is convex. Then
\begin{enumerate}
\item[(iii)]  $\| \cdot \|_B$ is a conic norm iff $B$ is convex
and iff $S$ is convex (its second fundamental form with respect to
the inner normal is positive semi-definite).

\item[(iv)] $\| \cdot \|_B$ is a conic norm with strict triangle
inequality iff $B$ is strictly convex and iff $S$ is strictly
convex (i.e. the hyperplane tangent to $S$ at each point $v_0$
only touches $B$ at $v_0$).

\end{enumerate}

\end{thm}

\begin{proof}
 The map  $\| \cdot \|_B$ is well defined from property
$\it (a1)$ (in fact, $\| v \|_B$ can be defined as the unique
$\lambda$ in $\it (a4)$ if $v\neq 0$). Then, it is positive
 and positively homogeneous. Its smoothness follows from
the smoothness of $S$ and the transversality ensured in $\it (a3)$,
as these conditions characterize when the
bijective map $\R ^+ \times S\rightarrow A\setminus\{0\},
(t,v)\mapsto tv$ is a diffeomorphism.  Moreover, $B$ is the
closed unit ball by construction and $\it (a2)$.

For  $(i)$, the implication to the right follows from part
$(ii)$ of Proposition \ref{plema}, and  the converse because the
compactness of $S$ implies that $A=V$.
Therefore, assuming that $S$ is a sphere, the properties $\it (a3)$ and
$\it (a4)$ (or the diffeomorphism given in the previous paragraph)
imply that $0$ belongs to the inner domain delimited by $S$.
Therefore, $B$ is  the inner domain (recall $\it (a2)$ and $\it (a3)$), and
$0\in \ri B$. For the continuity, if $\{x_n\}$ converges to $0$
but $\|x_n\|_B$ does not, 
we can assume that, up to a subsequence,  $\|x_n\|_B$  converges to
some $c\in(0,+\infty]$. Then, $x_{n}/\|x_{n}\|$ also converges to
some $y\in S$ up to a subsequence,
and $x_{n}\to c y\not=0$, which gives a contradiction.

For $(ii)$, recall that positive homogeneity implies
that $1$ must be a regular value of $\| \cdot \|$ and, so, $S$ is
a closed smooth hypersurface in $A$. Moreover, the opposite $\xi$
of the vector position is transverse to $S$ and points out to $\ri B$.
 Putting
$G=\frac 12\|\cdot\|_B^2$  and using  \eqref{esconv} (recall
Remark \ref{r26}), $\sigma^\xi(X,X)>0$ if and only if $g(X,X)>0$
for every vector field $X$ tangent to $S$. This statement together
with  last statement in Proposition \ref{fundamentalprop} (recall
again Remark \ref{r26})  prove $(ii)$.

 For $(iii)$ and $(iv)$, even though the equivalences between
the (strict or not) convexities of  $B$ and $S$ are known (see for
example \cite{BCGS, Sa} and references therein),  we will prove
the full cyclic implications for the sake of completeness.
Moreover, the first implications to the right in $(iii)$ and
$(iv)$ are straightforward from the triangle inequality applied to
$ \lambda u + (1-\lambda)v$. For the second implication to the
right, consider a straight  line $\mathbf r$ tangent to $S$ at
some $v_0\in S$ with direction $v$,  and let $(a,b)\subset\R$ be
the interval of points $t\in \R$ such that $v_0+t v\in A$ (recall
that $A$ is convex). Then the function $f:(a,b)\rightarrow \R$,
given by $f(t)=G(v_0+t v)-1$, where $G=\frac 12\|\cdot\|^2_B$,
satisfies $f(0)=0$, $\dot f(0)=0$  (as $dG_{v_0}(v)=0$) and
$\sigma^\xi_{v_0}(v,v)= -$ Hess$(v,v)=-\ddot f(0)$ (recall
\eqref{esconv0} and \eqref{nablacanonica} above). Now consider the
plane $\pi={\rm span}\{v_0,v\}$. If $\ddot f(0)<0$, then the
intersection $\pi\cap S$ must  lie, close to $v_0$, in the
connected component determined by the  line $\mathbf r$ that does
not contain $0$. Easily, this  contradicts the convexity of $B$
and proves the implication in part $(iii)$.
  For $(iv)$, if $\mathbf r$ touches $\pi\cap S$ in two points,
  then $S$ must contain a whole segment by convexity and it cannot be strictly convex. Reasoning with all the tangent right lines to $S$ in $v_0$ we conclude that the tangent hyperplane in $v_0$ touches only $v_0$.

To close the cyclic implications  in $(iii)$ and $(iv)$, we have
to prove that if $S$ is (strictly) convex, then $\|\cdot\|_B$
satisfies the (strict) triangle inequality. Given $u,v\in B$
(linearly independent), consider $\tilde{u}=u/\|u\|_B$,
$\tilde{v}=v/\|v\|_B\in S$ and define $\phi:[0,1]\rightarrow \R$
as
\[\phi(t)=G(t \tilde{u}+(1-t) \tilde{v})=G(\tilde{v}+t(\tilde{u}-\tilde{v})).\]
If we denote $y=\tilde{v}+t(\tilde{u}-\tilde{v})$, then $\ddot\phi(t)=g_y(\tilde{u}-\tilde{v},\tilde{u}-\tilde{v})\geq 0$ for $t\in (0,1)$ and $\phi(0)=\phi(1)=1$. It is easy to prove that $\phi$ is constantly equal to $1$ or $\phi(t)<1$ for $t\in (0,1)$.
In particular, for $t=\|u\|_B/(\| u \|_B+ \| v \|_B)$ we obtain
$$
z=\frac{\| u \|_B}{\| u \|_B+ \| v
\|_B}\tilde{u} + \frac{\| v \|_B}{\| u \|_B+ \| v
\|_B}\tilde{v} \in B.
$$
Then, $1\geq\| z \|_B =\| u +v\|_B/ (\| u \|_B+ \| v \|_B)$, as required. Regarding the strict inequality, observe that if $S$ is strictly convex, $\phi$ cannot be constantly equal to $1$.

\end{proof}
\begin{rem}\label{rojo} In the previous theorem, if the hypothesis $\it (a4)$ of
Proposition \ref{pfunda} were not imposed, then $\| \cdot \|_B$
would not be positive.  That is,   $\| v \|_B\geq 0$ would
still hold for all $v\in A$, but the equality could be reached
for some $v\neq 0$ and, thus, for all its {\em degenerate}
direction $D_v$. This could happen even if $A=V$ and, in this
case, $B$ would not be compact (nor $S$ homeomorphic to a sphere).
Even though this possibility could be also admitted, we prefer not
to include it. Or, equally, if degenerate directions appear, $A$
is supposed to be redefined  in order  to exclude them.
\end{rem}
Finally, we consider in particular the case of pseudo-norms.

\begin{prop}\label{ppseudi}
 Let $\| \cdot \|:V\rightarrow \R$ be a Minkowski
pseudo-norm. Then
\begin{enumerate}[(i)]
\item  the affine open  forward  (resp. backward)
balls
 constitute  a basis
for the natural topology of $V$,

\item if $\| \cdot
\|$ is not a Minkowski norm, the fundamental tensor field  $g$ is
degenerate at some direction.
\end{enumerate}
\end{prop}
\begin{proof}   Choose an auxiliary Euclidean norm $\| \cdot
\|^E$ and  denote $S^E$ its unit sphere.

For $(i)$, the compactness of the unit spheres of  $\| \cdot \|^E$
and $\| \cdot \|$ implies that each ball for one (pseudo-)norm
contains  small balls for the other one.

For $(ii)$, it is well-known that, at any  point $v_0$ in $S$
which is a relative maximum for the $\parallel \cdot
\parallel^E$-distance, the second fundamental form $\sigma^\xi$
with respect to $\xi=-v_0$ must be  positive definite (see
\cite[Ch VII, Prop. 4.6]{KN}). Thus, so is $g$ at $v_0$ (recall
Eq. \eqref{esconv} and Remark \ref{r26} together with last statement
of Proposition \ref{fundamentalprop}) and, as we assume
that  $\| \cdot \|$ is not a Minkowski norm, there must exist
$v_1\in V$ where $g$ is degenerate (see also \cite{Lovas}).

\end{proof}

\begin{rem} \label{rmpseudonorm}
These two properties cannot be extended to Minkowski conic norms
---this is trivial for $(ii)$, and see Example  \ref{ex3}
below for $(i)$.
\end{rem}

\subsection{Some simple examples}\label{sexamples} Theorem \ref{tfunda} provides a simple
picture of Minkows\-ki norms and their generalizations in terms of
the closed unit ball $B$ and the indicatrix $S$. Next, we consider $\R^2$
 and construct
illustrative examples of the generalizations of Minkowski
norms stated above ---easily, these examples can be  extended to higher
dimensional vector spaces.

Consider a curve $c: I \subset \R \rightarrow \R^2$,
$c(\theta)=r(\theta)(\cos \theta , \sin \theta )$, which does not
cross the origin and it is smoothly parameterized by its angular
polar coordinate $\theta \in I$. For each $\theta\in \R$ consider
the radial half line
$$
l_\theta=\{ r(\cos \theta, \sin \theta): r>0 \}.
$$
Recall that $c$ crosses transversally all $l_\theta$ with $\theta
\in I$, and let $n(\theta)$ be the normal vector characterized by
$\langle n(\theta),\dot c(\theta)\rangle_0 =0, -\langle n(\theta),
c(\theta)\rangle_0 >0$, where  $\langle \cdot ,\cdot \rangle_0$ is
the natural scalar product on $\R^2$. Put also $S=c(I)$ and
\be\label{hatg} \hat g(\theta)= \langle \ddot c(\theta),
n(\theta)\rangle_0 =\frac{2\dot r^2+r(r-\ddot r)}{\sqrt{r^2+\dot
r^2}}. \ee
 Recall that the sign of $\hat g(\theta)$ is equal to the sign of
the second fundamental form of $S$ at $c(\theta)$ with respect to
the direction $\xi=-c(\theta)$, and it is also equal to the sign
of the fundamental tensor field $g$ on the non-zero vectors of
$T_{c(\theta)}S$.

Let us construct then some examples.

\begin{exe}\label{ex1}
Assume that $c$ is a smooth closed curve, say, $I=[0,2\pi]$, and
$c$, as well as all its derivatives, agrees at $0$ and at $2\pi$.
By
 Theorem
\ref{tfunda} the interior region $B$ of $S$ is the unit ball of a
Minkowski pseudo-norm $\|\cdot\|_B$. Clearly $\hat g(\theta)\geq
0$ everywhere iff $S$ is  convex and $\|\cdot\|_B$ becomes a norm;
in this case, if $\hat g(\theta)$ is equal to zero only at
isolated points then $S$ is strictly convex and $\|\cdot\|_B$
satisfies the strict triangle inequality.
 The strict inequality $\hat g >0$ characterizes when
 $\|\cdot\|_B$ is a
Minkowski norm.
\end{exe}

\begin{exe}\label{ex2}
Assume  that $I=(\theta_-,\theta_+)$ with $\theta_+-\theta_-<\pi$.
 Now, $A=\cup_{\theta\in I}l_\theta$ constitutes a  convex
conic domain of $\R^2$, and take the closed subset $B$ of $A$
delimited by $l_{\theta_-}, l_{\theta_+}$ and $S$. Again, $B$ is
the closed unit ball for a Minkowski conic pseudo-norm
 $\|\cdot\|_B$, which will be a Minkowski conic norm
 iff $\hat g>0$  (recall that all the items of Theorem \ref{tfunda} will be
 applicable).

 If $c$ admits a smooth extension to $\theta_-$ or $\theta_+$ which
 is transverse to the corresponding $l_\theta$, then
 $\|\cdot\|_B$ can be regarded as the restriction to
 $A$ of a Minkowski conic pseudo-norm defined on a bigger domain
 $\tilde A$.   Recall that,  if $S$ is convex, $c$ can always be extended in $\R^2$ to  either $\theta_-$ or $\theta_+$
as the tangent line at any $\theta_0 \in I$ must intersect either
$l_{\theta_-}$ or $l_{\theta_+}$, but this extension may be
 non-transverse, or even the point zero (put $I=(0,\pi/2)$ and choose $c$ as
 the $\theta$-reparametrization of the  branch $x>x_0$ of the parabola $y=(x-x_0)^2$ for some $x_0\geq
 0$).

Observe that, by construction, $ S$ is always closed as a subset
of $A$. If $S$ is also closed as a subset of $V$ then $S$ cannot
be convex everywhere (essentially, $l_{\theta_-}$ and
$l_{\theta_+}$ would be parallel to asymptotes of $c$). In this
case, it would be also natural to extend
$\|\cdot\|_B$ to all $V$, so that the directions
 $l_{\theta_\pm}$, away from $A$, were degenerate (see part $(1)$ of Remark
\ref{rojo}).
\end{exe}
\begin{exe}\label{ex2185}
Let $c(\theta)= \theta (\cos\theta,\sin\theta)$, $\theta\in
(\epsilon,2\pi-\epsilon)$ for some $\epsilon>0$. Clearly, this
curve is convex (use \eqref{hatg}) and determines a Minkowski
conic norm $\|\cdot\|$. For small $\epsilon$, $\|\cdot\|$ has a
non-convex domain $A$ and satisfies: (i) it is not the restriction
to $A$ of a Minkowski norm on all $\R^2$, and (ii) there are
vectors $u,v\in A$ such that $u+v\in A$ but they do not satisfy
the triangle inequality (choose $u=(0,2\sin(2\varepsilon)),
v=(\cos (2\epsilon),-\sin( 2\varepsilon))$. Then $\|u\|=4\sin (2\varepsilon)/\pi$, $\|v\|=1/(\pi-2\varepsilon)$ and $\|u+v\|=1/(2\varepsilon)$. It is clear that
when $\varepsilon$ is small enough, $\|u\|+\|v\|<\|u+v\|$).
\end{exe}

\begin{exe}\label{ex2a}
Consider now two examples with $A=\{(x,y)\in \R^2: y>0\}$.
First, the curve $c_1$ obtained as the reparametrization of a
branch of the parabola $t\mapsto (t,\sqrt{1-t})$, $t\in (-\infty,1)$,
 with the angular coordinate $\theta\in (0,\pi)$. In this case
$\hat{g}>0$, the curve defines a Minkowski conic norm and
Proposition \ref{opennormballs} is applicable.  Second, the curve
$c_2$ obtaining as reparametrization of the straight line $y=1$
(i.e., the associated norm is $\|(x,y)\|:= y$). This is a conic
norm, but the forward and backward balls do not constitute a
subbasis for the topology of $\R^2$.  Moreover, all its
admissible curves connecting two fixed points will have the same
length and, according to Subsection \ref{s2d}, they will be
extremal of the energy functional (geodesics) when reparameterized
at constant speed.
\end{exe}

\begin{exe}\label{ex3}
 As a refinement of the previous example for future
referencing, consider the parabola $c(t)= (t,1-t^2), t\in \R$
which can be reparameterized with the angular coordinate $\theta
\in (-\pi/2,3\pi/2)$. This curve is strongly convex and defines a
Minkowski conic norm with only a direction excluded from the domain
---concretely, $A=\R^2\setminus (l_{-\pi /2}\cup\{0\}). $
 Its affine open balls (see \eqref{eopenballs})
centered at any $v_0=(x_0,y_0)$ are delimited
by some parabola as follows:
$$
B^+_{v_0}(r)= \{(x,y): (x,y)-v_0\in A;\,
y-y_0<r\left(1-\frac{(x-x_0)^2}{r^2}\right)\}.
$$
In this example (as well as in the first one in Example
\ref{ex2a}), the balls are always open for the topology of $V$
but, as two such parabolas always intersect, the topology
generated by the balls on $\R^2$ is strictly coarser than the
usual one (in particular, no pair of points are Hausdorff related
for that topology). Thus, the result in part $(i)$ of Proposition
\ref{ppseudi} cannot be extended to the case of Minkowski conic
norms.
\end{exe}

\begin{exe}\label{ex4}
Consider the curve $c(t)= (\sinh t,\cosh t ), t\in \R$, which can
be reparametrized by the angular coordinate $\theta \in
(\pi/4,3\pi/4)$ and, thus, can be regarded as a particular case of
Example \ref{ex2}. This curve defines a conic pseudo-norm
$\parallel\cdot\parallel$ with $S$ concave everywhere (i.e. $\hat
g<0$), which can be interpreted as follows.

Consider the natural Lorentzian scalar product on $\R^2$:
$$
\langle (x,y),(x',y')\rangle_1 =x x' -y y'.
$$
Then, $A$ is composed of all the vectors $v=(x,y)$ which are {\em
timelike} ($\langle (x,y),(x,y)\rangle_1 <0$) and {\em
future-directed} ($y>0$); moreover:
$$
\| v \|
=\sqrt{-\langle (x,y),(x,y)\rangle_1}.
$$ Here, the (forward) affine balls
centered at the origin can be described as:
$$
B^+_0(r)= \{v\in A: -\langle v,v\rangle_1 < r^2\}, $$ which is the
open region delimited by a hyperbola and the lines $y=\pm x$.

Remarkable,  the triangle inequality does not hold, but a reverse
strict triangle inequality does, namely:
$$
\| v+w\| \; \geq \; \| v \| + \| w
\|, \quad \quad \forall v,w\in A
$$
with equality iff $w=\lambda v$ for some $\lambda>0$ (this and the
following property are well-known for Lorentzian scalar products,
see \cite[Proposition 5.30]{Oneill83}, for example). By applying
this triangle inequality, the following property follows easily.
For all $v_0\in A$ and $\epsilon>0$ there exists a sequence
$P_0=0, P_1, \dots, P_k=v_0$ such that each $v_i=P_i-P_{i-1}$
belongs to $A$ and $\sum_{i=1}^k \| v_i \| < \epsilon$. That is,
there are poligonal curves connecting 0 and $v_0$ with arbitrarily
short $\| \cdot \|$-length.

 Finally, recall also that
$\|\cdot\|$ can be naturally extended to {\em past-directed}
timelike vectors, or even all the vectors  where the Lorentzian
scalar product does not vanish, yielding a bigger conic domain
$\tilde A$ with four connected parts. Such a situation becomes
natural for {\em quadratic Finsler} manifolds (see part (1) of
Remark \ref{r1} below).

\end{exe}

\section{Pseudo-Finsler and conic Finsler metrics}\label{s2}

\subsection{Notion} First, let us  clarify the notions of
 Finsler metric to be used.
\begin{defi}\label{d1}
Let $M$ be a  manifold and $A$ an open subset of the tangent
bundle $TM$ such that $\pi(A)=M$, where $\pi:TM\rightarrow M$ is
the natural projection, and let $F:A\rightarrow
[0,\infty)$ be a continuous function. Assume that $(A,F)$ satisfies: 
\begin{enumerate}[(i)]
\item  $A$ is conic in $TM$, i.e., for every $v\in A$ and
$\lambda>0$, $\lambda v\in A$ ---or, equivalently, for each $p\in
M$, $A_p:=A\cap T_pM$ is a conic domain in $T_pM$.
\item $F$ is
smooth on $A$ except at most on the zero vectors.
\end{enumerate}
 We say that $(A,F)$, or simply $F$, is a {\em conic pseudo-Finsler metric} if  each
restriction $F_p:=F|_{A_p}$ is a Minkowski conic pseudo-norm on
$T_pM$. In this case, $F$ is

\begin{enumerate}[(i)]
\item[(iii)]  a {\em conic Finsler metric} if
each $F_p$ is a Minkowski conic norm, i.e. the {\em fundamental
tensor} $g$ on $A\setminus\{$zero section$\}$ induced by all the
fundamental tensor fields $g^{(p)}$ at each $ p\in M$ on
$A_p\setminus \{0\}$ is positive definite for all $p\in M$,
\item[(iv)] a {\em pseudo-Finsler metric} if $A=TM$, i.e.,
each $F_p$ is a Minkowski pseudo-norm so that $A_p=T_pM$ for all
$p\in M$,
\item[(v)] a (standard) {\em  Finsler metric} if $F$
is both, conic Finsler and pseudo-Finsler, i.e. $A=TM$ and $g$ is
pointwise positive definite.
\end{enumerate}

\end{defi}
Observe that the fundamental tensor $g$ can be thought as a
section of a  fiber bundle over $A$. To be more precise,
denote also as $\pi:A\setminus\{0\}\rightarrow M$ the
restriction of the natural projection from the tangent bundle to
$M$,  $TM^*$ the cotangent bundle of $M$ and
$\tilde{\pi}:TM^*\rightarrow M$ the natural projection. Define
$\tilde{\pi}^*:\pi^*(TM^*)\rightarrow A\setminus\{0\}$ as the
fiber bundle obtained as the pulled-back bundle of
$\tilde{\pi}:TM^*\rightarrow M$ through
$\pi:A\setminus\{0\}\rightarrow M$: \be \label{efibrados}
\xymatrix{
\pi^*(TM^*)\ar[d]_{\tilde{\pi}^*}&TM^*\ar[d]^{\tilde{\pi}}\\
A\setminus\{0\}\ar[r]_{\pi}&M\, }\ee Then $g$ is a smooth
symmetric section of the fiber  bundle $\pi^*(TM^*)\otimes
\pi^*(TM^*)$ over $A\setminus\{0\}$. Let us remark that if we fix
a vector $v\in A\setminus\{0\}$, then $g_v$ is a symmetric
bilinear form on $T_{\pi(v)}M$. From now on this will be the
preferred notation.

\begin{rem}\label{r1} Some comments are in order:
\begin{enumerate}[(1)]
\item Our definition of  (standard) Finsler metric agrees with
classical references as \cite{AP,Akbar-Zadeh06,BaChSh00,ChSh05, Mo06,shen2001,
Sh01} and our definition of conic Finsler metric is equivalent to the notion of {\em generalized Finsler metric} by \cite{Br} (but we retain our nomenclature as we are also dealing with other generalizations).  According to our nomenclature, a classical Kropina metric
is a {\em conic Finsler} metric (see Corollary \ref{cKropina}
below).
 A  Matsumoto metric has a
maximal conic domain where it is {\em conic Finsler} and a bigger
one where it is {\em conic pseudo-Finsler} (see Corollary
\ref{matsustrongly} below). Another way to generalize Finsler
metrics is considering what we would call a {\em quadratic Finsler
metric} $L$ (here $L$ would be positively homogeneous of degree 2,
and the definition would include all the Lorentzian or
semi-Riemannian metrics). This approach can be found in
\cite{AIM93,Asa85, Mat86,Per08}.
In principle, the case of conic quadratic Finsler metrics is more
general as one can always  consider the conic quadratic Finsler
metric  $L=F^2$  associated to any conic pseudo-Finsler metric.
The converse holds only when $L$ is positive away from the zero
section (otherwise, one can define just a conic pseudo-Finsler
metric on the positive-definite directions). Nevertheless, it is
natural to assume then some minimum restrictions which essentially
reduces to our case\footnote{\ See, for example, the definition of
Finsler spacetime in \cite{PW}. Its associated Finsler metric
satisfies all the assumptions of a conic pseudo-Finsler metric
(except strict positive definiteness) plus other additional
hypotheses.}. So, we will focus on the already general
conic pseudo-Finsler case, which allows one to clarify some properties of the distance and balls.
Let us finally point out that, in reference \cite{BeFa2000}, the authors define
a Finsler metric as what we call a conic Finsler metric and a pseudo-Finsler metric as a
conic quadratic Finsler metric with nondegenerate fundamental tensor.

\item Minkowski conic pseudo-norms as those studied in
Subsection \ref{sexamples}
give the first examples of conic pseudo-Finsler metrics. Starting
at them, one can yield examples of more general situations. For
instance, consider the following two conic pseudo Finsler metrics
on $\R^2$:
\begin{enumerate}[(a)]
\item  $A=\{v\in T\R^2: dy(v)>0\}$ and $F=dy|_A$,

\item $A=\{v\in T\R^2: dx(v)\neq 0, dy(v)\neq 0\}$ and
$F=\sqrt{dx^2}+\sqrt{dy^2}$.
\end{enumerate}

Recall that they come from Minkowski conic pseudo-norms with
non-compact and convex indicatrices, which are not strongly convex
at any point (so that only the non-strict triangle inequality will
 hold). It is not difficult to accept that a
 conic pseudo-Finsler metric $F^a$ on $\R^2$ can be defined such that
$F^a_{(x,y)}$ behaves as the metric in item (a) for $|y|< 1$,
as the conic Minkowski norm in Example \ref{ex3} for $y<-2$ and as
the Minkowski pseudo-norm in the Example \ref{ex4} for $y>2$.
Analogously, a Finsler metric $F^b$ on $\R^2$ can be defined such
that $F^b_{(x,y)}$ behaves as the metric in (b) for $x^2+y^2< 1$
and as the Finsler metric associated to the usual scalar product
on $\R^2$ for $x^2 +y^2>2$. In particular, the conic domain $A_p$
may vary from point to point so that such domains are not
homeomorphic.

\item  As the last example shows,  the number of connected
components of $A_p$ may vary with $p$, and $A$ may be connected
even if some $A_p$ are not. Recall that for a Minkowski conic
pseudo-norm $\parallel\cdot\parallel$, the conical domain $A$
might have (infinitely) many connected parts, even though we
consider typically the case $A$ connected. Accordingly,  a natural
hypothesis in the conic pseudo-Finsler case is to assume that $A$
is {\em pointwise connected} (i.e., any $A_p$, $p\in M$, is
connected) and, in the conic Finsler case,  we also may assume
that $A$ is {\em pointwise convex} ($A_p$ is a
 convex subset, and then connected, of each $T_pM$) in order to have the triangle
 inequality.

\item Typically, one can also   assume that the pair $(A,F)$ is
{\em maximal}, i.e., such that no pair $(\tilde A, \tilde F)$
extends $(A,F)$. This means that no conic pseudo-Finsler metric
$(\tilde A, \tilde F)$ satisfies $A\varsubsetneq \tilde A$ with
$\tilde F|_{A}=F$ (say, as in Examples \ref{ex3} or \ref{ex4}).
Zorn's lemma ensures the existence of maximal extensions. However,
as such an extension may be highly non-unique (recall Example
\ref{ex2} or the first one in Example \ref{ex2a}), and maximality
will not be assumed a priori along this paper. Moreover,
non-maximal pairs $(A,F)$ may be useful to model diverse
situations ---for example, to represent restrictions to the
possible
 velocities on relativistic particles which may move on $M$.


\item As a direct consequence of part $(ii)$ of Proposition
\ref{ppseudi}, {\em the fundamental tensor $g$ of a pseudo-Finsler
but not Finsler manifold must be degenerate at some points} ,
concretely, it must be degenerate on some directions at all the
tangent spaces $T_pM$, $p\in M$, where $g^{(p)}$ is not
definite positive.

\end{enumerate}
\end{rem}
Some properties of  conic pseudo-Finsler metrics can be reduced to
the case $A=TM$ taking into account the following result.

\begin{prop}\label{punidad}
Let $F: A\rightarrow \R$ be a conic pseudo-Finsler metric and
let $C\subset A$ be such that $C\cup \{\text{zero section}\}$ is a closed conic subset
of $TM$. Then, there exists a pseudo-Finsler metric $\tilde F$ on $M$
such that $\tilde F|_{C}=F|_{C}$.
\end{prop}
\begin{proof}
Consider some auxiliary Riemannian metric $g_R$ on $M$, let $F_R=\sqrt{g_R}$
be its associated Finsler metric, and take its unit sphere bundle
$S_RM\subset TM$. Let $\{U,V\}$ be the open covering of $S_RM$
defined as  $U=S_RM\cap (TM\setminus C)$ and $V=S_RM\cap A$.
Consider  a partition of the unity  $\{\mu_U, \mu_V\}$
subordinated to this covering, and regard these functions as
functions on $TM\setminus\{$zero section$\}$ just by making
them homogeneous of degree $0$. The pseudo-Finsler metric $\tilde F=
\mu_U F_R + \mu_V F$ satisfies the required properties.
\end{proof}

\begin{rem} \label{runidad}
Recall that, in the previous proposition, the functions which
constitute the partition of the unity are not defined on $M$ but
on $TM$. So, when applied to a conic Finsler metric, the obtained
extension is only a conic pseudo-Finsler metric (see Example
\ref{ex2185}). \end{rem}

\subsection{Admissible curves and Finslerian separation}  In this subsection, $M$ will be assumed to be connected in order to
avoid trivialities. Consider a smooth curve $\alpha$ with
velocity $\dot \alpha$ in a conic pseudo-Finsler manifold $(M,F)$.
As the expression $F(\dot\alpha)$ does not always make sense, we
need to restrict to curves where it does.
\begin{defi}
Let $F:A\rightarrow [0,\infty)$ be a conic pseudo-Finsler metric
on $M$. A piecewise smooth curve $\alpha:[a,b]\rightarrow M$ is
{\em $F$-admissible} (or simply {\em admissible}) if the right and
left derivatives $\dot\alpha_+(t)$ and $\dot\alpha_-(t)$ belong to
$A$ for every $t\in [a,b]$. In this case, the {\em ($F$-)length}
of $\alpha$ is defined as
\begin{equation}\label{lengthfunctional}
\ell_F(\alpha)=\int_a^b F(\dot\alpha(t))\df s.
\end{equation}
\end{defi}

\begin{defi}
Let $(M,F)$ be a conic pseudo-Finsler manifold, and $p,q\in M$. We
say that $p$ {\em precedes} $q$, denoted $p\prec q$, if there
exists an admissible curve  from $p$ to $q$. Accordingly, the {\it
future} and the {\it past} of $p$
 are the subsets
 \[
 \begin{array}{ccc}
 {\rm I}^+(p)=\{q\in M: p\prec  q\}&\text{and}&
 {\rm I}^-(p)=\{q\in M: q \prec  p\},
\end{array}
 \]
respectively. Moreover, we define ${\rm I}^+=\{(p,q)\in M\times M:
p\prec  q\}$, i.e., ${\rm I}^+$ is equal to the binary relation $\prec
$ as a point subset of $M\times M$.

\end{defi}
 It is straightforward to check the following properties.

\begin{prop}\label{pimas} Given a conic pseudo-Finsler manifold $(M,F)$:
\begin{enumerate}[(i)]
\item The binary relation $\prec $ is transitive.

\item The subsets ${\rm I}^+(p)$ and ${\rm I}^-(p)$ are open in $M$
for all $p\in M$, and ${\rm I}^+$ is open in $M\times M$.

\item If $F$ is a pseudo-Finsler metric then $\prec$ is trivial,
i.e.,  $p\prec  q$ for all $p,q$  in the (connected) manifold
$M$.
\end{enumerate}
\end{prop}
(For $(ii)$, recall the techniques in the proof of Proposition
\ref{openballs} below).

Now, we can introduce a generalization of the so-called
Finsler distance for Finsler manifolds.
\begin{defi}
Let $(M,F)$ be a conic pseudo-Finsler manifold $p,q\in M$ and
 $C_{p,q}^F$ the set of all the $F$-admissible
piecewise smooth curves from $p$ to $q$.  The {\em Finslerian
separation}  from $p$ to $q$ is defined as:
\[{\dist}_F(p,q)=\inf_{\alpha\in C_{p,q}^F}\ell_F(\alpha)\in [0,\infty].\]
\end{defi}
From the definitions, one has directly: \begin{equation}
\label{sep1} C_{p,q}^F\neq \emptyset \Leftrightarrow p\prec   q,
\quad {\dist}_F(p,q)=\infty \Leftrightarrow p\not\prec   q, \quad
{\dist}_F(p,q)\geq 0,\end{equation} as well as the triangle
inequality \begin{equation} \label{sep2} \dist_F(p,q)\leq
\dist_F(p,z)+\dist_F(z,q), \quad \forall p,q,z \in
M.\end{equation}

Analogously, we can define two kinds of balls, the  {\it (open)
forward $d_F$-balls}  $B_F^+(p,r)$ and the {\it backward} ones
$B_F^-(p,r)$, namely,
\[B_F^+(p,r)=\{q\in M: \dist_F(p,q)<r\}, \quad
B^-_F(p,r)=\{q\in M:\dist_F(q,p)<r\}.\]

\begin{prop}\label{openballs} Let $(M,F)$ be a conic pseudo-Finsler manifold.
Then the open forward and backward $d_F$-balls are open subsets.
\end{prop}

\begin{proof}
Let $q\in B_F^+(p,r)$ and $\alpha: [0,r_0]\rightarrow M$  be an
 admissible curve from $p$ to $q$ parametrized by arclength, so that $r_0<r$ and
$\dot\alpha(r_0)$ belongs to $A\cap T_qM$. Let $K$ be a compact
neighborhood of $\dot\alpha(r_0)$ in the unit ball at $T_qM$
entirely contained in  $A$, and $C$ the corresponding conical
subset $C=\{tv: v\in K, t>0\}$. Choosing coordinates $(U,\varphi)$
in a neighborhood $U$ of $q$, $TU$ can be written as
$\varphi(U)\times \R^n$ and $C$ is identified with some
$\{\varphi(q)\}\times C_0$, namely $C_0=\{tu: u\in K_0, t>0\}$ is
a conical subset of $\R^n$, and the compact subset $K_0$ is
defined by $ d\varphi(K)=\{q\}\times K_0$. Moreover, choosing a
smaller $U$ if necessary (say, with compact closure), we can
assume $\varphi(U)\times C_0\subset d\varphi(A)$ and $F\circ
\varphi^{-1}$ is bounded on $\varphi(U)\times K_0$ by some
constant $N_0\geq 1$. So, choose some $r_-<r_0$ with
$r_0-r_-<(r-r_0)/N_0$, such that $\alpha([r_-,r_0])\subset U$,
 the segment from $\varphi(\alpha(r_-))$ to $\varphi(\alpha(r_0))$ is
contained in $U$ and
there
exists $\lambda>1$ with
$\frac{\lambda}{r_0-r_-}(\varphi(\alpha(r_0))-\varphi(\alpha(r_-)))\in
K_0$. (Observe that we can always find such an $r_-$ because
\[\lim_{r_-\to r_0}
\frac{\varphi(\alpha(r_0))-\varphi(\alpha(r_-))}{r_0-r_-}=d\varphi(\dot\alpha(r_0))\]
and $d\varphi(\dot\alpha(r_0))\in K_0$).  The segments in
$\varphi(U)$ (regarded as a subset of $\R^n$) which start at
$\alpha(r_-)$ with velocity in $K_0$ and defined on some interval
$[r_-, b]$ with $b<r_-+(r-r_0)/N_0]$ have $F$-length smaller than
$r-r_0$ and cover a neighborhood $\varphi(W)$ of $\varphi(q)$. So,
the neighborhood $W\ni q$ is clearly contained in the required
ball. The proof for backward balls is analogous.
\end{proof}
\begin{defi} A conic pseudo-Finsler metric $F$ is {\em (Riemannianly)
lower bounded} if  there exists a Riemannian metric
$g_0$ on $M$  such that $F(v)\geq \sqrt{g_0(v,v)}$ for
every $v\in A$. In this case,  $g_0$ is called a {\em
(Riemannian) lower bound} of $F$.
\end{defi}
\begin{rem}\label{rprevio}
Several comments are in order:
\begin{enumerate}[(1)]
\item[(1)] The lower boundedness  needs to be checked
only locally, that is, {\em a conic pseudo-Finsler metric $F$ is
lower bounded if and only if for every point $p\in M$, there
exists a neighborhood $U\subset M$ of $p$ and a Riemannian metric
 $g$ on $U$ such that $F(v)\geq\sqrt{g(v,v)}$ for every $v\in
A\cap TU$.}
 In fact, each $U$ can be chosen compact and, then, all
the metrics $g$ can be taken homothetic to a fix auxiliary one
$g_R$ defined on all $M$. By using paracompactness,   a locally
finite open covering $\{U_i\}_{i\in I}$ of $M$ exists such that
 $F(v)\geq \frac{1}{c_i}\sqrt{g_R(v,v)}$ for
some $c_i\geq 1$ and all $v\in A\cap TU_i$. If $\{\mu_i\}_{i\in
I}$ is a partition of the unity subordinated to $\{U_i\}_{i\in
I}$, then $g_0=\frac{1}{\sum_{i\in I}\mu_ic_i}g$ is the required
lower bound of $F$.

\item[(2)] The lower boundedness of $F$ can be also checked by looking at
the indicatrix $S_p\subset A\cap T_pM, p\in M$ for each $F_p$.
Indeed, {\em $F$ is lower bounded if and only if for each $p\in M$
there exists a compact neighborhood $\hat K$ of  $0\in T_pM$ in
$TM$ such that $S_q\subset \hat K$ for all $q\in K:=\pi(\hat K)$.}

\item[(3)] Either from the definition or from the criteria above,
it is straightforward to check that  classical Kropina and
Matsumoto metrics  (see Subsection \ref{s3b} for their definition)
are lower bounded.

\item[(4)] Observe that the distance $d_{g_0}$ associated to the
lower bound $g_0$ satisfies $d_{g_0}(p,q)\leq d_F(p,q)$ for every
$p,q\in M$ and, then, $ B^+_F(p,\epsilon)\subseteq
B_{g_0}(p,\epsilon)$.
\end{enumerate}
\end{rem}

\begin{lemma}\label{basis}
Let $F:A\rightarrow [0,+\infty)$ be a conic pseudo-Finsler metric
in $M$ that admits a lower bound $g_0$. Fix $p\in M$ and $r>0$.
 Given $x\in B_{g_0}(p,r)$, there exist $q, q'\in B_{g_0}(p,r)$
 and $\epsilon>0$ such
that $x\in B^+_F(q,\epsilon)\cap B^-_F(q',\epsilon)$ and
$B^+_F(q,\epsilon)\cup B^-_F(q',\epsilon)\subseteq
B_{g_0}(p,r)$.
\end{lemma}
\begin{proof}
Consider an admissible curve $\gamma$ passing through $x$.
Clearly, we can choose a point $q$ on $\gamma$  such that
$d_F(q,x)<\epsilon$ and $B_{g_0}(q,\epsilon)\subset B_{g_0}(p,r)$.
So (recall part (4) of Remark \ref{rprevio}),  $x\in
B^+_F(q,\epsilon)\subseteq B_{g_0}(q,\epsilon)\subseteq
B_{g_0}(p,r)$. By reversing the parametrization of $\gamma$,  the
required point $q'$ can be found analogously.
\end{proof}

\begin{prop}\label{topologia}
Let $F$ be a lower bounded conic pseudo-Finsler metric on $M$.
Then the collection of subsets $\{B^+_F(p,r): p\in M, r>0\}$
(resp. $\{B^-_F(p,r): p\in M, r>0\}$) constitute a topological
basis of $M$.
\end{prop}
\begin{proof}
It is a direct consequence of Proposition \ref{openballs} and Lemma \ref{basis}.
\end{proof}

\subsection{Discussion on balls and distances}
 In
general, the
  Finslerian separation may behave in a very different way from a
  distance. To make more precise this statement, let us introduce the
  following notion, see \cite[Section 1]{Bus44}  and also \cite[page 5]{Z}  and \cite{FHS}.

  \begin{defi}\label{dgd}
  A {\em generalized distance} $d$ on a set $X$ is a map $d:X\times X\rightarrow \R$
  satisfying  the following
axioms:
\begin{itemize}
\item[(a1)] $d(x,y)\geq 0$ for all $x,y\in X$.
\item[(a2)] $d(x,y)=d(y,x)=0$ if and only if $x=y$.
\item[(a3)] $d(x,z)\leq
d(x,y)+d(y,z)$ for all $x,y,z\in X$.
\item[(a4)] Given a sequence
$\{x_n\}\subset X$ and $x\in X$, then $\lim_{n\rightarrow
\infty}d(x_n,x)=0$ if and only if $\lim_{n\rightarrow
\infty}d(x,x_n)=0$.
\end{itemize}
  \end{defi}
  For any generalized distance,  the {\em forward}
  balls have a natural meaning, and generate a topology which
  coincides with the one generated by the  {\em backward} balls.
Recall that the  Finslerian  distance associated to a Finsler
manifold is a generalized distance in the sense above.  This
can be extended to pseudo-Finsler manifolds, as they are
Riemannianly lower bounded.
\begin{thm}\label{tballs}
For any pseudo-Finsler manifold $(M,F)$, the Finslerian separation
$d_F$ is a generalized distance,  the forward (resp. backward)
open balls generate the topology of the manifold, and $d_F$ is
continuous with this topology.
\end{thm}

\begin{proof} The consistency of the definition, including the axiom
(a1) in Definition \ref{dgd} follows
from part $(iii)$ of Proposition \ref{pimas}, and the axiom
(a3)  from Eq. \eqref{sep2}. For the other two axioms, choose
two auxiliary Riemannian metrics $g_1, g_2$ with distances, resp.,
$d_1, d_2$ which satisfy $g_1(v,v)\leq F^2(v)\leq g_2(v,v)$ for
all $v$ in $TM$ (easily, $g_1$ and $g_2$ can be chosen conformal
to any prescribed Riemannian metric $g_R$). Obviously:
$$d_1 \leq d_F \leq d_2 .$$ From this equality,  (a2) and (a4)
plus the assertion of the topology of the manifold, follow
directly. The last assertion is general for the topology
associated to any generalized distance \cite[p. 6]{Z}.
\end{proof}

\begin{rem}\label{pseudoaffineballs}
However,  if a vector space $V$  endowed with a pseudo-norm
$\parallel \cdot \parallel$ is regarded as a pseudo-Finsler space,
the affine and $d_F$-balls may differ. More precisely, the
unit (forward) affine ball is always contained in the
corresponding unit (forward) $d_F$-one, and the inclusion is
strict if $g^{p}$ is indefinite at some $p\in
V\backslash\{0\}$.
\end{rem}
Theorem \ref{tballs} shows that some properties of $d_F$ in the
Finsler case are retained in the pseudo-Finsler one. However,
in the general conic case the things may be very different, as
only Proposition \ref{openballs} can be claimed. As the following
example suggests, the relation $\prec$ may give a tidier
information.

\begin{exe}\label{ex317}
(1) In Example \ref{ex4}, regarded as a conic pseudo-Finsler
manifold $(M,F)$, the following properties occur:
$${\rm I}^+(0)=\{(x,y): |x|<y\}, \quad d_F(0,v)=\left\{
\begin{array} {cll}0 & \hbox{if} & v\in {\rm I}^+(0), \\ \infty & \hbox{if} & v\not\in
{\rm I}^+(0).\end{array} \right.$$
In particular, $d_F(0,0)=\infty$ and the open $d_F$-balls
look very different to the open affine balls.

(2) Moreover, one can obtain a quotient conic pseudo-Finsler
metric on the torus $T^2=\R^2 / \Z^2$  just taking into account
that $F$ is invariant under translations. In this quotient, {\em
the Finslerian separation between any two points is equal to 0}.

(3) Regarding Example \ref{ex3} as a conic Finsler metric (with
non pointwise convex $A$), one also has $d_F((0,0),(0,-s))=0$ for
all $s\geq 0$ (recall that such $(0,-s)$ does not belong to
$A_{(0,0)}$ but $(0,0)\prec (0,-s)$ as one can connect these two
points by means of $F$-admissible piecewise smooth curves with a
break). Thus, $\{(0,-s): s\geq 0\} \subset B_F^+((0,0),r)$ but
$\{(0,-s): s\geq 0\} \cap B^+_{(0,0)}(r)=\emptyset$ for all $r>0$.
\end{exe}
 Let us see that even in the conic Finsler case, the Finslerian
separation can be discontinuous  around two points
$p,q$ with  $d_F(p,q)<\infty$.
\begin{exe}\label{discontinuity}
  Consider  the translation $\varphi:\R^2\rightarrow \R^2$ given by
$\varphi(x,y)=(x+2,y)$ and   the group of isometries
$G=\{\varphi^n: n\in\Z\}$. The quotient  $\R^2/G$ can be
identified with the strip $\{(x,y)\in\R^2:-1\leq x\leq 1$ and
$(-1,y)\equiv (1,y)$ $\forall y\in \R\}$. Now consider the convex
open cone  of $T_{(0,0)}\R^2$ determined by the vectors $(-1,3)$
and $(1,1)$ and choose in every point of the cylinder the open
cone obtained as the parallel traslation of this cone. This
defines the subset $A$. The Finsler metric $F$ is the square root
of the Euclidean metric restricted to $A$. Finally observe that
the  separation from  $p=(0,0)$ is not continuous in $q=(-2/3,2)$.
In fact, $d_F((0,0),(-2/3,2))=2\sqrt{13}/3$ (which is equal to the
Euclidean distance from $(0,0)$ to $(4/3,2)$), while
$d_F((0,0),(-2/3+\epsilon,2))=\sqrt{(\epsilon-2/3)^2+4}$ (the
Euclidean distance between the two points) for $0<\epsilon<2$.
\end{exe}

\subsection{Minimization properties of geodesics in conic Finsler
metrics}\label{s2d} For any conic pseudo-Finsler metric we can
define the energy functional as
\begin{equation*}\label{energyfunctional}
E_F(\alpha)=\int_a^b F(\dot\alpha(t))^2\df s,
\end{equation*}
where $\alpha:[a,b]\rightarrow M$ is any $F$-admissible piecewise
 smooth curve between two fixed points $p,q\in M$. Then,
geodesics can be defined as critical points of this functional. Of
course, they will be also critical points of the length funcional
in \eqref{lengthfunctional}, but as critical points of the energy
functional, geodesics are obliged to have constant speed.

When the fundamental tensor $g$ is non-degenerate we can define
the  {\em Chern connection} (see \cite[Chapter 2]{BaChSh00}),
which is a connection for  $ \tilde{\pi}^*:
\pi^*(TM)\rightarrow A\setminus\{0\}$ (recall the diagram in
\eqref{efibrados}).
Moreover, fixing a vector field $W$ on an open $U\subset M$, the
Chern connection gives an affine connection on $U$ that we will
denote by $\nabla^W$. If we fix a vector field $T$ along a curve
$\gamma$, the Chern connection provides a covariant derivative
along $\gamma$ with reference $T$ that we will denote as $D^T$
(see \cite[Sections 7.2 and 7.3]{shen2001}). In this case,
geodesics are uniquely determined when the initial conditions are
given as the solutions of the equation $D^T_TT=0$, where
$T=\dot\gamma$, and there is a unique geodesic tangent to a given
vector of the tangent bundle. Moreover, we can define the {\em
exponential map} in $p\in M$, $\exp_p:\Omega\subseteq A_p
(\subseteq T_pM)\rightarrow M$, as $\exp_p(v)=\gamma(1)$,
where $\gamma$ is the unique geodesic such that $\gamma(0)=p$ and
$\dot\gamma(0)=v$, whenever $\gamma$ is defined at least in
$[0,1]$.

\begin{conv}
According to our conventions, if $0\not\in A_p$, the constant
curve $\gamma_p(t)=p$ for all $t\in \R$ is not an admissible curve
and, so, it is not a geodesic (this situation is common in our
study, as it happens in some points whenever a conic
pseudo-Finsler metric is not pseudo-Finsler). Nevertheless, when
considering any curve $c: [a,b]\rightarrow M$ starting at $p$, we
will not care about whether this initial point is obtained from
the exponential. So, we will say that $c$ is contained in the
image by $\exp_p$ of some subset $S\subseteq \Omega (\subseteq
A_p)$ as the geodesic balls below (or that  $c$ lies in $\exp_p
(S)$),  if $c((a,b])\subset \exp_p(S)$,
(that is, we will work as if $\exp_p(0)=p$ when forced by the
context) and we will assume that $\dot c(a)\in A_p$ (otherwise, the curve
would not be admissible).
\end{conv}

 When the fundamental tensor is degenerate, the interpretation
of geodesics as solutions of $D^T_TT=0$ makes no sense, but as
critical points of the energy functional still holds. However,
then geodesics are not univocally determined by its velocity at
one point (see \cite[Example 3 (Fig.1)]{Mat}; as a limit case,
all the constant-speed parameterized curves in the second case of
Example \ref{ex2a} would be geodesics).
 Furthermore, if the fundamental tensor is not
degenerate but indefinite at some $v\in A_p$, the geodesic with
velocity $v$ is  never minimizing. Next, we will see that
for conic Finsler metrics, geodesics minimize in geodesics balls,
but a previous technical discussion is required.

\begin{rem}\label{gausslemma}
The Gauss Lemma in typical references such as \cite[Lemma
6.1.1]{BaChSh00} is proven only for standard Finsler metrics.
Nevertheless, its local nature allows one  to prove it in a much
more general context: it works for any conic pseudo-Finsler metric
whenever the exponential is defined in a  neighborhood
${\mathcal W}$ of $v\in A_p\subset T_pM$  and $g$ is
non-degenerate in ${\mathcal W}$. In fact, put $r=F(v)(>0)$,
let $\gamma$ be a geodesic with $\dot \gamma (0)=v$, and $w$, a
tangent
vector to the 
sphere
$S_0^+(r)=\{u\in A_p: F(u)=r\}$. Choose a curve
$\rho:(-\epsilon,\epsilon)\rightarrow S^+_0(r) \cap \mathcal{W}
\subset A_p$ such
that $\dot \rho(0)=w$, and consider the 
variation
$(-\epsilon,\epsilon)\times [0,r]\rightarrow A_p$, $(s,t)\mapsto
t\rho(s)$. Accordingly, the variation of $\gamma (=\gamma_0)$ by
the geodesics $\gamma_s(t)=\exp_p(t\rho(s))$ with variational
field $U$ satisfies
\[\frac{\partial}{\partial s} \ell_F(\gamma_s )|_{s=0}=\frac{1}{F(T)}g_T(U,T)|_0^{r}-\int_0^r g_T(U,D^T_T(T/F(T)))\df t,\]
where $T=\dot\gamma$  (see \cite[Exercise 5.2.4]{BaChSh00}). As
 $\gamma$ is a geodesic, and all the curves in the variation have the same length and depart
from the same point, the last equation
 reduces to  the Gauss Lemma, i.e.:
 $ g_T(d\exp_{p}[w],T)=0$.
\end{rem}

\begin{prop}\label{minimizeconic}
Let $(M,F)$ be a conic Finsler metric with pointwise convex
$A$, and assume that $\exp_p$
is defined  on a certain 
ball $B^+_0(r)$ of
$(T_pM,F_p)$ and it is a diffeomorphism onto  the geodesic ball
$B_p^+(r):=\exp_p(B^+_0(r))$.  Then, for any $q\in B_p^+(r)$ the
radial geodesic from $p$ to $q$ is, up to reparametrizations, the
unique minimizer of the Finslerian separation among the admissible curves
contained in $B_p^+(r)$.
\end{prop}
\begin{proof}
Given a point $q\in B_p^+(r)$, let $\tilde{r}$ be the length of
the radial geodesic from $p$ to $q$. Consider any smooth
admissible curve $c:[0,1]\rightarrow M$ from $p$ to $q$ contained
in $B_p^+(r)$. As $\exp_p$ is a diffemorphism on $B^+_0(r)$, there
exist   functions $s:[0,1]\rightarrow [0,r)$ and
$v:[0,1]\rightarrow S_0^+(1)\subset T_pM$ 
uniquely determined   (both smooth up to zero and $s$
continuous at zero) such that $c(u)=\exp_p(s(u)v(u))$ so that
$s(0)=0$ and $s(1)=\tilde{r}$. Then, define $\sigma: [0,r) \times [0,1]\rightarrow M$ as
$\sigma(t,u)=\exp_p(t v(u))$ and denote $T=\frac{\partial
\sigma}{\partial t}$ and $U=\frac{\partial \sigma}{\partial
u}=d\exp_p[t\dot v]$. We can express $c(u)=\sigma(s(u),u)$ for
every $u\in [0,1]$. Hence by the chain rule
\begin{equation}\label{cprime}\dot c(u)=\dot s(u)\, T+U.
\end{equation}
Moreover, by the Fundamental Inequality in \eqref{fundineq},
$g_T(T,\dot c)\leq F(T)F(\dot c)$, and using the previous identity
we get
\[\dot s\, g_T(T,T)+g_T(T,U)\leq F(T) F(\dot c), \quad \quad \forall s\in (0,1].\]
By Gauss Lemma (see  Remark \ref{gausslemma} above), $g_T(T,U)=0$,
and $F(T)=1$ because $T$ is the velocity of a geodesic with
initial velocity equal to $v(u)$ (recall that $F(v(u))=1$).
Moreover, $g_T(T,T)=F(T)^2=1$ and, therefore,  $\dot s\leq F(\dot
c)$. Applying the last inequality, we deduce that
\[\ell_F(c)=\int_0^1 F(\dot c)du\geq  \lim_{\epsilon \searrow 0} \int_\epsilon^1  \dot s(u)du= s(1)-s(0)=\tilde{r}.\]
As the length of the radial geodesic is exactly $\tilde{r}$, this
shows that it is a global minimizer between all the curves lying
in $B^+_p(r)$.

For the unicity up to reparametrization,  observe that the
equality in the Fundamental Inequality can happen just when
$T$ and $\dot c$ are proportional, but  looking at Eq.
\eqref{cprime}, this means that $U=0$ (recall that, by Gauss
Lemma, $U$ is $g_T$-orthogonal to $T$). As $\exp_p$ is a
diffeomorphism, $U=0$ implies that $\dot v=0$, and therefore, that
$c$ is a reparametrization of a radial geodesic.

\end{proof}

\begin{rem}\label{rm0}
 Regarding the absolute minimization:
\begin{enumerate}[(1)]
\item
In the standard Finsler case, the radial geodesic is an absolute
minimizer for the curves from $p$ to $q$ on all $M$. In fact, if
we consider a curve $\beta$ that goes out of $B^+_p(r)$, then its
length must also be greater or equal than $ r$, as the portion of
the curve in $B^+_p(r)$ from $p$ to the boundary of $B^+_p(r)$ has
already length greater or equal to $r$. Summing up, in a Finsler
metric, each radial geodesic segment of a geodesic ball as in
Proposition \ref{minimizeconic} is the unique curve in $M$ of
minimum  length  which connects its endpoints, up to affine
reparameterizations.

\item
This absolute minimization cannot be ensured in the conic Finsler
case. The reason is that shorter curves which cross the boundary
of $\exp_p(A_p)$ may appear. In fact, consider the following
example on $\R^2$. Let $p=(0,0), q=(0,1) \in \R^2$ and
$R_\epsilon$ the open rectangle of vertexes
$V_p^\pm(\epsilon)=p+(\pm \epsilon,-\epsilon)$,
$V_q^\pm(\epsilon)=q+(\pm \epsilon,\epsilon)$ and choose small
$\epsilon_1>\epsilon_0>0$. Consider a Riemannian metric $g=\Lambda
\langle\cdot,\cdot\rangle$  (conformal to the usual one
$\langle\cdot,\cdot\rangle$)  such that the conformal factor
$\Lambda (>0)$ is equal to some big constant $B$ on
$R_{\epsilon_0}$ and to 1 outside $R_{\epsilon_1}$. $B$ is chosen
so that there are curves $y\mapsto (x(y),y)$ from $p$ to $q$ which
go outside $R_{\epsilon_1}$ and have $g$-length much smaller than
$B$ (which is the $g$-length of the segment from $p$ to $q$). Now,
let $F$ be the conic Finsler metric associated to $g$ with domain
$A$ determined as follows. $A_{(x,y)}\equiv (x,y)+C_y$ where $C_y$
is: (a) the open cone delimited by $p$ and the vertexes
$V_q^\pm(\epsilon_o)$ for $(x,y)$ with $y\leq 0$, (b) the half
plane $y>0$ for $(x,y)$ with $y\geq \delta$ where $\delta\geq 0$
is  some prescribed small constant, and (c) a cone obtained by
opening the one in (a) until the half plane in (b) for
$0<y<\delta$. (Notice that the choice $\delta=0$ is permitted, and
then the case (c) will not occur, however, the possibility to
choose $\delta
>0$ stresses that a  choice of $C_y$ discontinuous on $y$, is
irrelevant here). Then, clearly $B_p^+(1+\epsilon_0)\subset
R_{\epsilon_0}$ but there are admissible curves from $p$ to $q$
shorter than the geodesic segment in the ball.

\item
As a consequence of (2),  some additional hypothesis must be
imposed in order to ensure the character of absolute minimizer for
radial geodesics. Taking into account (1), a sufficient hypothesis
would be: {\em the boundary of $B_p^+(r)$ in ${\rm I}^+(p)$ is equal to
$\exp_p(S_0^+(r))$.}
\end{enumerate}
\end{rem}

\begin{rem}\label{rm1}
 Regarding the domain:
\begin{enumerate}[(1)]
\item
 In Proposition \ref{minimizeconic} we have  assumed that the
domain A is pointwise convex. Even though this could be weakened
(only the fundamental inequality was required, and this holds
under more general hypotheses, see Remark \ref{fundineqRem}), if A
is not pointwise convex, some other pathologies may  happen. For
example,
for the conic Finsler norm of Example \ref{ex3},
   the separation from $(0,0)$ to any
point of the form $(0,-s)$, $s\geq 0$ is $0$, but there is no
geodesic joining them (recall part (3) of Example \ref{ex317}).
The underlying reason is that these points $(0,-s)$ lie in ${\rm
I}^+(0,0)$ but not in the affine ball $B_0^+(r)$ for any $r\geq
0$.

\item
In contrast with the behavior of Example \ref{ex3},
Proposition \ref{minimizeconic} plus part (3) of Remark \ref{rm1} imply:
{\em in a vector space $V$ endowed with a conic Minkowski norm
defined on a convex domain $A$, the $d_F$-balls of the Finslerian
separation coincide with the affine balls.} Indeed, the convexity
of each $A_p=A\cap T_pV$ implies now that $\exp_p(A_p)$ coincides
with ${\rm I}^+(p)$.

\item
As a further improvement, to
find general conditions on a conic Finsler metric $F$ so that
Proposition \ref{minimizeconic} can be applied for small enough
affine balls would  be very interesting. Proposition
\ref{topologia} suggests that the lower boundedness of $F$ might
be such a property. Recall that when the Finsler metric is not
lower bounded, the forward  balls may not constitute a topological
basis (see Example \ref{ex3} and the first one in \ref{ex2a}), and
it is not difficult to find cases where the exponential is not
defined even in small balls: consider the first Example \ref{ex2a}
(notice that $A$ is pointwise convex in this case), remove from
$\R^2$ the semiaxis $\{(x,0): x\leq 0\}$ and consider any ball
$B^+_{(x_0,y_0)}(r), r>0$ for $(x_0,y_0)$ with $y_0<0$.
\end{enumerate}

\end{rem}

 Other properties of local minimization and conjugate points of
geodesics in the conic Finsler case deserve to be studied,
although we will not go through them.

\subsection{Summary on geodesics and distance}
Let us summarize and compare in this subsection the properties of
geodesics and distance of the different types of Finsler metric
generalizations.

Recall first the following three well-known properties of a
Finsler manifold:
\begin{enumerate}[(i)]
\item
$d_F$ is a generalized distance, $d_F$ is continuous, and the open
forward  (resp. backward)  balls generate the topology of $M$.

\item
The equation of the geodesics is well defined and characterize
them univocally from initial data (the velocity at a point).
Moreover, the radial geodesic segments of a geodesic ball are
strict global minimizers  of the energy functional (see part (1)
of Remark \ref{rm0}).

\item
In particular, for a Minkowski norm the affine balls
and the $d_F$-balls agree,  and the geodesics as a Finsler
manifold coincide with affinely parametrized straight lines for a
Minkowski norm.

\end{enumerate}
Let us analyze how these three properties behave in the generalizations
of Finsler metrics we are studying.

 For a pseudo-Finsler non-Finsler manifold:
\begin{enumerate}[(1)]
\item Properties in (i) hold (see Theorem \ref{tballs}).

\item Properties in (ii) do not hold. In fact, the equation of the
geodesics is necessarily ill defined (as $g^{(p)}$ must be
degenerate at some directions, see part (5) of  Remark \ref{r1}).
Moreover, even in the directions where this equation is well
defined, geodesics do not minimize in any sense in general
(recall, for example, Remark \ref{pseudoaffineballs}). The
definition of geodesics as critical curves of the energy
functional makes sense, but their uniqueness from the initial
velocity does not hold in general (second Example \ref{ex2a}).

\item The first assertion in Property (iii) does not hold for
Minkowski pseudo-norms whenever $g$ is indefinite (see Remark
\ref{pseudoaffineballs}).  A straightforward computation shows
that the affine  lines are geodesics for any pseudo-norm, but
additional critical curves of the energy functional may appear
---in the case of norms, the uniqueness of the affine lines as
geodesics hold if the strict triangle inequality
hold.

\end{enumerate}

 For a conic Finsler manifold:
\begin{enumerate}[(1)]
\item The two first assertions in Property (i) do not hold in
general, as $d_F$ may reach the value $\infty$ and also can
be discontinuous in finite values (see Example \ref{discontinuity}).
The open forward balls do not
constitute a basis for the topology in general (first Example
\ref{ex2a} and \ref{ex3}) even though they do in particular cases
(Proposition \ref{topologia}).

\item The first assertion in Property  (ii)  holds, in the sense
that the equation of the geodesics is always well defined along
the admisible directions. The radial geodesic segments minimize
inside geodesic balls (see Proposition \ref{minimizeconic} and
also Remark \ref{rm0}, especially its part (2)).

\item  Properties in (iii) hold essentially, but  some subtleties
must be taken into account. For a conic Minkowski norm {\em with
convex $A$} the affine and $d_F$ balls coincide (see  part (2),
and also (1),  in Remark \ref{rm1}). The affinely parametrized
curves with velocity in $A$ coincide with the geodesics as a conic
Finsler manifold.

\end{enumerate}
 For a conic pseudo-Finsler manifold in general the previous
 properties do not hold, but Propositions \ref{openballs} and \ref{topologia} and relation $\prec$  may
yield some general information, which may be useful combined with
particular properties of different classes of examples.

\section{Constructing conic Finsler metrics}\label{ls3}

New, let us consider  the case when a new conic pseudo-Finsler
metric is constructed from a homogeneous combination of
pre-existing conic Finsler metrics and one-forms, as a
generalization (in several senses) of the known $(\alpha,\beta)$-metrics.

\subsection{General result and first consequences}\label{s3a}
 In the following,
$F_1,\ldots,F_n$ will denote conic Finsler metrics on the manifold
$M$  of dimension $N$, with fundamental tensors $g^1,\dots, g^n$ (so that $g^k_v$ is
the fundamental tensor of $F_k$ at the tangent vector $v$). Moreover,
the so-called {\em angular metrics} (see \cite[Eq. 3.10.1]{BaChSh00}) are defined
as
\begin{equation}\label{angularmetric}
h_v^k(w,w)=g_v^k(w,w)-\frac{1}{F_k(v)^2}g_v^k(v,w)^2,
\end{equation}
for any $v\in A\setminus 0$ and $w\in T_{\pi(v)}M$ and
$k=1,\ldots,n$. Due to Cauchy-Schwarz inequality, the angular
metric of a conic Finsler metric is always positive semi-definite
being the direction of $v$ the only degenerate direction.

The
intersection $$A=\cap_{k=1}^n A_k \subset TM$$ of their conic
domains $A_k$ is assumed to be non-empty at each point, i.e.,
$\pi(A)=M$. Moreover, $\beta_{n+1},\beta_{n+2},\ldots,\beta_{n+m}$
will denote $m$ one-forms on $M$. The indexes $k,l$ will run from
$1$ to $n$. The indexes which label the ordering of one-forms will
be denoted with Greek letters $\mu, \nu$ and run from $n+1$ to
$n+m$, while the indexes $r, s$ will run from $1$ to $n+m$.

Let $B$ be a conic open subset of $\R^{n+m}$ and consider a
continuous  function $L: B\times M\rightarrow \R$, which
satisfies:
\begin{itemize}
\item[(a)] $L$ is smooth and positive away from $0$, i.e, on $(B\times
M) \setminus (\{0\}\times M)$.

\item[(b)] $L$ is $B$-positively homogeneous of degree 2, i.e.,
$L(\lambda x, p)=\lambda^2 L(x,p)$ for all $\lambda>0$ and all
$(x,p)\in B\times M$.
\end{itemize}
Denote by ${\rm Hess}(L)$ the $B$-Hessian function matrix
associated to $L$, that is, the matrix  with coefficients the
functions $a_{rs}=L_{,rs}$, 
where the
comma denotes derivative with respect to the corresponding
coordinates of $\R^{n+m}$ (in particular, $L_{,rs}$ means the
second partial derivative of $L$ with respect to the $r$-th and
$s$-th variables). We say that ${\mathrm Hess}(L)$ is positive
semidefinite if so is the  matrix of functions at each $(x,p)\in
B\times M$.
Finally, consider the function $F^2: A\subseteq TM\rightarrow \R$
defined as: \be \label{ef2}
F^2(v)=L(F_1(v),F_2(v),\ldots,F_n(v),\beta_{n+1}(v),\beta_{n+2}(v),\ldots,\beta_{n+m}(v),\pi(v)).\ee
\begin{thm}\label{central}
For any $L$ satisfying (a) and (b) as above, and $F^2$ as in
\eqref{ef2}, the function $F:=\sqrt{F^2}$ 
is a conic pseudo-Finsler metric with domain $A$ and fundamental
tensor:
\begin{equation}\label{fundamentalTensor}
 g_v(w,w)=\frac 12\sum_k\frac{L_{,k}}{F_k(v)}h_v^k(w,w)+\frac 12
 P(v,w){\mathrm Hess}(L) P(v,w)^T,
\end{equation}
for all $v\in A\setminus{0}$,  $w\in T_{\pi(v)}M$, where the
superscript $^T$ denotes transpose and $P(v,w)$ the $(n+m)$-tuple
\be \label{epvw}
P(v,w)=(\frac{g_v^1(v,w)}{F_1(v)}, 
\ldots,\frac{g_v^n(v,w)}{F_n(v)},\beta_{n+1}(w),\ldots,\beta_{n+m}(w))
\quad  (\in \R^{n+m}),\ee  and the derivatives  $L_{,r},
L_{,rs},...$  of $L$ are computed on \be
\label{epbm}(F_1(v),F_2(v),\ldots,F_n(v),\beta_{n+1}(v),\beta_{n+2}(v),\ldots,\beta_{n+m}(v),\pi(v))
\; \in B\times M. \ee Moreover,  $g$ is positive semi-definite if
the following conditions hold on the points in \eqref{epbm}
obtained by taking any $v\in A\setminus{0}$:
\begin{itemize}
\item[(A)] $L_{,k}\geq 0$ for every $k=1,\ldots,n$, and \item[(B)]
${\mathrm Hess}(L)$ is positive semi-definite,
\end{itemize}
and it is positive definite (i.e., $F$ is a conic Finsler metric)
if, additionally:
\begin{itemize}
\item[(C)] $L_{,1}+\ldots+L_{,n}>0$.
\end{itemize}
\end{thm}
\begin{proof}
Recall from Proposition \ref{fundamentalprop} that $g_v(v,w)=\frac
12\frac{\partial}{\partial s} F(v+sw)^2|_{s=0}$. As, clearly,
\[\left. \frac{\partial F_k(v+sw)}{\partial s}\right|_{s=0}=\left.\frac{\partial \sqrt{F_k(v+sw)^2}}{\partial s}\right|_{s=0}=\frac{1}{F_k(v)} g_v^k(v,w)\]
and $\left.\frac{\partial \beta_\mu(v+sw)}{\partial
s}\right|_{s=0}=\beta_\mu(w)$, we have
\begin{equation}\label{laanterior}
2 g_v(v,w)=\left.\frac{\partial F^2(v+sw)}{\partial s}\right|_{s=0}=\sum_k \frac{1}{F_k(v)} g^k_v(v,w) L_{,k}+\sum_\mu \beta_\mu(w)L_{,\mu}.
\end{equation}
Again 
from Proposition
\ref{fundamentalprop},
\[g_v(w,w)=\frac12 \left. \frac{\partial}{\partial t}\right|_{t=0} \left(  \frac{\partial }{\partial s}\left. F^2((v+tw)+sw)\right|_{s=0}
\right)=  \frac{\partial }{\partial t} \left.
g_{v+tw}(v+tw,w)\right|_{t=0}
\] and, then,  applying \eqref{laanterior}, we obtain
\begin{multline}
 2 g_v(w,w) 
 =\sum_k\frac{L_{,k}}{F_k(v)}\left(g_v^k(w,w)-\frac{1}{F_k(v)^2}g^k_v(v,w)^2\right) \\ + \sum_{k,l}\frac{L_{,kl}}{ F_k(v)F_l(v)}g_v^k(v,w)g_v^l(v,w)
+2\sum_{k,\mu}\frac{L_{,k\mu}}{F_k(v)}g_v^k(v,w)\beta_\mu(w)\\+\sum_{\mu,\nu}L_{,\mu\nu}
\beta_\mu(w)\beta_\nu(w).\label{fundamentalTensor2}
\end{multline}
So, the expression of $g_v$ in \eqref{fundamentalTensor} follows
 by using
 $F_k(v)^2=g^k_v(v,v)$, formula \eqref{angularmetric}, and that  the two last lines of
the formula above can be written as $P(v,w){\mathrm Hess}(L)
P(v,w)^T$ for $P(v,w)$ as in \eqref{epvw}.

For its positive semi-definiteness, recall that applying
hypothesis (A), plus  the fact that the angular metrics $h_v^k$ in
\eqref{angularmetric} are positive semi-definite:
$$Q_1(v,w):=
\sum_k  \frac{L_{,k}}{F_k(v)}h^k_v(w,w)
 \geq 0. $$ 
 So, the result follows by applying (B) to obtain:
$$2g_v(w,w)-Q_1(v,w)= P(v,w) {\mathrm Hess}(L)  P(v,w)^T\geq 0.$$

Finally, for the  strict positiveness of $g$, if $w=\lambda v$,
with $\lambda\not=0$, then $g_v(w,w)=\lambda^2
F^2(v)>0$ and, otherwise, 
 $Q_1(v,w)>0$ by hypothesis (C) and because $h^k_v$ is degenerate
 only in $v$.
\end{proof}
\begin{rem}
 Observe that the expression of the fundamental tensor in
\eqref{fundtensorsum} still holds when $F_1,\ldots,F_n$ are just
conic pseudo-Finsler metrics. Indeed, we can obtain more accurate
conditions to construct Finsler metrics. For example, we can
consider conic pseudo-Finsler metrics with positive semi-definite
fundamental tensor satisfying $(A)$ and $(B)$ and to check that
for every $w\in T_{\pi(v)}M, v\in A\setminus 0$, there is some
term strictly positive.
\end{rem}
Next, let us consider two particular cases. The first one is very
elementary, and we include it because, as far as we know, it does
not appear in the classic books on the subject.

\begin{cor}\label{sum}
 Let $F_1,F_2,\ldots,F_n$ be conic Finsler metrics on $M$ defined on the same conic domain $A$.
 Then $F=F_1+F_2+\cdots+F_n$ is also a conic Finsler metric and its fundamental tensor
 is given as
\begin{equation}\label{fundtensorsum}
g_v(w,w)=F(v)\sum_{k}
\frac{h_v^k(w,w)}{F_k(v)}
 +\left(\sum_k\frac{g^k_v(v,w)}{F_k(v)}\right)^2,
\end{equation}
for any $v\in A\setminus{0}$ and $w\in T_{\pi(v)}M$.
\end{cor}
\begin{proof}
Apply Theorem \ref{central} to $L:\R^n\times M\rightarrow \R$,
defined as $L((x_1,\dots, x_n),p)= (\sum_r x_r)^2$ (recall
$L_{,k}=2(\sum_r x_r)$, thus $L_{,kl}=2$), and apply Eq.
\eqref{fundamentalTensor}.
\end{proof}
As a generalization to be compared with the results on Randers
metrics below:
\begin{cor}\label{rrandersr}
Consider $n$ conic Finsler metrics and $m$ one-forms as above.
Then,
\[R(v):=\left(\sum_{k=1}^n
F_k(v)^{q}+\sum_{\mu=k+1}^{m+n}
|\beta_\mu(v)|^{q}\right)^{\frac{1}{q}}\] is a conic
pseudo-Finsler metric with  domain $A$ if $q\geq 2$, and with
domain
\[A_R:=A\setminus\{v\in TM: \beta_\mu(v)= 0 \text{ for some $\mu=m+1,\ldots,m+n$}\},\]
for  $2>q\geq 1$. At each case, the  fundamental tensor $g$ is
given by:
\begin{align}\label{fundtensorsquare}
R(v)^{2q-2}g_v(w,w)=&  R(v)^q \sum_k F_k(v)^{q-2}
h_v^k(w,w)\nonumber\\& + \frac
12(q-1)\sum_{k,l}(F_k(v)F_l(v))^q\left(\frac{g_v^k(v,w)}{F_k(v)^2}-\frac{g_v^l(v,w)}{F_l(v)^2}\right)^2
\nonumber\\& + \frac
12(q-1)\sum_{\mu,\nu}|\beta_\mu(v)\beta_\nu(v)|^q\left(\frac{\beta_\mu(w)}{\beta_\mu(v)}-\frac{\beta_\nu(w)}{\beta_\nu(v)}\right)^2
\nonumber\\& +  (q-1) \sum_{k,\mu} |F_k(v)\beta_\mu(v)|^q
\left(\frac{g_v^k(v,w)}{F_k(v)^2}-\frac{\beta_\mu(w)}{\beta_\mu(v)}\right)^2
\nonumber\\& +\left(\sum_k F_k(v)^{q-2}g_v^k(v,w)+\sum_\mu
|\beta_\mu(v)|^{q-2}  \beta_\mu(v)  \beta_\mu(w)  \right)^2,
\end{align}
for all $v\neq 0$ in the conic domain and $w\in T_{\pi(v)}M$.

As a consequence,  $g$ is always positive semi-definite and, if
$n\geq 1$, $R$ is a conic Finsler metric.
\end{cor}
\begin{proof}
Consider the function $L: B \times M\rightarrow \R$ with
\[B=\R^{n+m}\setminus\{(x_1,\dots , x_{n+m}): x_r=0  \text{ for some } r=1,\ldots,n+m\}\]
 defined as
\[L\left((a_1,a_2,\ldots,a_{n+m}),p\right)=\sqrt[q]{(|a_1|^q+|a_2|^q+\cdots+|a_{n+m}|^q)^2},\]
Let us denote $U=|a_1|^q+|a_2|^q+\cdots+|a_{n+m}|^q.$ Then
$L_{,r}=2 a_r |a_r|^{q-2} \, U^{\frac{2}{q}-1}$ and
\[L_{,rs}=2\delta_{rs}  (q-1) |a_r|^{q-2} \, U^{ \frac{2}{q}-1}
+2(2-q) a_ra_s |a_ra_s|^{q-2} U^{ \frac{2}{q}-2},\]
where $\delta_{rs}$ is  Kronecker's delta. Observe that if
$x=(x_1,x_2,\ldots,x_{n+m})\in\R^{n+m}$,
\[x{\rm Hess}(L) x^T=(q-1)U^{\frac{2}{q}-2}\sum_{r,s}|a_ra_s|^q\left(\frac{x_r}{a_r}-\frac{x_s}{a_s}\right)^2
+2U^{\frac{2}{q}-2}(\sum_r x_r  a_r|a_r|^{q-2}  )^2.\] The
expression of the fundamental tensor \eqref{fundtensorsquare}
follows easily by using the last identity to compute
\eqref{fundamentalTensor}
---and, then, the other assertions follow directly.
\end{proof}
\begin{rem}\label{consequences}
Last corollaries can be  useful in different situations. For
example,   let $F$ be a conic Finsler metric  on some open subset
$U$ of $M$, and $C$  a closed subset of $M$ included in $U$. If
$\tilde F$ is any conic Finsler metric on $M\setminus C$, by using
a partition of the unity, there exists a conic Finsler metric
$F^*$ defined on all $M$ which extends $F$ on $C$ and $\tilde F$
on $M\backslash U$ (compare with Proposition \ref{punidad} and
Remark \ref{runidad}). The following corollary develops a
different application for the isometry group ${\rm Iso}(M,F)$ of a
Finsler manifold.
\end{rem}

\begin{cor} \label{cisom} Let $F$ be a non-reversible Finsler metric. Then,
\begin{align*}
\tilde{F}(v)&=F(v)+F(-v),&
\hat{F}(v)&=\sqrt{F(v)^2+F(-v)^2},
\end{align*}
for $v\in TM$, are reversible  Finsler metrics and:
\[{\rm Iso}(M,F) \subseteq {\rm Iso}(M,\tilde{F})\cap {\rm Iso}(M,\hat{F}).\]
 Moreover, this inclusion becomes an equality for the
connected components of the identity of each side.
\end{cor}
\begin{proof}
The inclusion ${\rm Iso}(M,F) \subseteq {\rm Iso}(M,\tilde{F})\cap
{\rm Iso}(M,\hat{F})$ is straightforward.    Clearly, it also
holds  for the connected components of the identity and, for the
converse, let $\phi$ belong to the connected component of the
right hand side. Recall that $$F(v)=\frac{1}{2}
\left(\tilde{F}(v)\pm \sqrt{2\hat{F}(v)^2-\tilde{F}(v)^2}\right)$$
with positive  sign  if $F(v)-F(-v)\geq 0$ and with  negative  sign otherwise.
Both expressions at  right hand (with  both  signs) are
$\phi$-invariant, and the radicand of the root is equal to
$\left({F}(v)-{F}(-v)\right)^2$. Therefore, $|F(d\phi (v))-F(-
d\phi (v))|= |F(v)-F(-v)|$ for all $v$, and the equality holds
even if the absolute values are removed, as $\phi$ belongs to the
connected part of the identity. Thus, $F(d\phi(v))=F(v)$ is a
consequence of the expression above for $F$ (applied on $v$ and
$d\phi(v)$) corresponding to the sign of $F(v)-F(-v)$.
\end{proof}

\begin{rem}  (1)  Notice that the strict inclusion can hold. In fact,
for any non-reversible Minkowski norm $\|\cdot\|\equiv F$ on a
vector space $V$, minus the identity belongs to $\left({\rm
Iso}(V,\tilde F)\cap {\rm Iso}(V,\hat F)\right)\setminus {\rm
Iso}(V,F)$.

(2) Corollary \ref{cisom} reduces  the proof  that the group of
 isometries of any  Finsler metric is a Lie group, to the
 reversible case (simplifying, say,  the proof in \cite[Theorem
 3.2]{DengHou02}, which modifies Myers-Steenrod  \cite{MySt39} technique for -symmetric- distances).
An alternative proof has been also developed recently in \cite[Section 2]{MRTZ09}
by proving that ${\rm Iso}(M,F)$ is always a closed subgroup of
the isometry group of a suitable average Riemannian metric, see
also \cite{ricardo}.

\end{rem}

\subsection{The case of $(F_0,\beta)$ metrics}\label{s3b}
 Next, we will focus on the
application of Theorem \ref{central}  on just one Finsler metric
$F_0$ with conic domain $A_0\subseteq TM$ and one 1-form $\beta$.
A very important case is that of $(\alpha,\beta)$-metrics
introduced by M. Matsumoto in \cite{Mat72} (see also
\cite{Bucataru,KAM95,MSS10,Mat89,Mat91b,Mat91c,Mat92,SSI08}) as a
generalization of Randers, Kropina and Matsumoto metrics. Here
$\alpha$ denotes the square root of a Riemannian metric and
$\beta$ a one-form on $M$.
In our general setting, we will consider  combinations of the form
$F_0+\beta$, $F_0^2/\beta$ and $\frac{F_0}{F_0-\beta}$, which
generalize the Randers, Kropina and Matsumoto
$(\alpha,\beta)$-cases (compare 
also with \cite{Has88}). Such combinations  were named {\em
$\beta$-changes} in \cite{Shi84} but, in concordance  with our
approach, they will be called here $(F_0,\beta)$-metrics, that is,
a conic pseudo-Finsler  metric $F$ is an {\em
$(F_0,\beta)$-metric} if it is obtained as a positively
homogeneous combination of a conic Finsler metric $F_0$ and a
one-form $\beta$.

In order to characterize the conic domains for the strong
convexity of an $(F_0,\beta)$-metric, we will give first a
necessary condition.
\begin{prop}\label{reciproco}
 Assume that the $(F_0,\beta)$-metric $F=\sqrt{L(F_0(v),\beta(v),x)}$ is a conic Finsler metric and the dimension of $M$ is $N>2$.
Then $L_{,1}>0$. 
\end{prop}
\begin{proof}
Let $v\in A\setminus 0$,
and denote as $z\in T_{\pi(v)}M$ the $g_v^0$-dual of
$\beta$, i.e. $\beta(w)=g_v^0(w,z)$ for all $w\in T_{\pi(v)}M$.
Consider an orthonormal basis $e_1,\ldots, e_N$ for $g_v^0$ such
that $e_1=v/F_0(v)$ and $z=\lambda^1 e_1+\lambda^2 e_2$ for some
$\lambda^1,\lambda^2\geq 0$. Putting
$w=\sum_{i=1}^N w^ie_i$ and using \eqref{fundamentalTensor}:
\begin{equation}\label{ejemejem}
2 g_v(w,w)= \frac{1}{F_0} L_{,1} \sum_{i=2}^N (w^i)^2 + L_{,11}
(w^1)^2  + 2L_{,12}  (\lambda^1 w^1+\lambda^2 w^2)w^1
+ L_{,22} (\lambda^1 w^1+\lambda^2 w^2)^2\\
\end{equation}
and the result follows by choosing $w=e_3$.

\end{proof}

\begin{rem}\label{rL1defposit2} Rewriting the terms in \eqref{ejemejem}
as $a (w^1)^2+2b w^1 w^2+c (w^2)^2+ d\sum_{i\geq 3} (w^i)^2$ the
metric $g_v$ is positive definite if and only if $d>0$, $c>0$ and
$ac-b^2> 0$, i.e., $L_{,1}>0$, $(\lambda^2)^2L_{,22}+\frac 1F_0
L_{,1}>0$ and
\[(L_{,11}+2\lambda^1 L_{,12}+(\lambda^1)^2 L_{,22})(L_{,22}(\lambda^2)^2+\frac 1F_0 L_{,1})-(\lambda^2)^2(L_{,12}+\lambda^1 L_{,22})^2 > 0.\]
  For $N=2$ the term in $d$ does not appear, and the  condition $L_{,1}>0$ is dropped.
\end{rem}

\subsubsection{Canonical form $F=F_0\cdot\phi(\beta/F_0)$} Next, for any $(F_0,\beta)$-metric $F=\sqrt{L(F_0,\beta)}$
we define $\phi(s)=\sqrt{L(1,s)}$ and, thus,
$F=F_0\cdot\phi(\beta/F_0)$. This change yieds convenient
expressions for $(F_0,\beta)$-metrics, including the Kropina,
Matsumoto and Randers ones. The next result is related to one by
Chern and Shen in \cite{ChSh05}, which is discussed as an Appendix
(Subsection \ref{sA}).
\begin{prop}\label{functionpsi}
Let $\phi:I\subseteq\R\rightarrow \R$ be a smooth function
 satisfying $\phi>0$. Given a conic Finsler metric
 $F_0:A_0\subseteq TM\rightarrow \R$ and a one-form $\beta$, consider
 the conic pseudo-Finsler metric
\[F(v)=F_0(v)\phi\left(\frac{\beta(v)}{F_0(v)}\right)\]
on  $A:=\{v\in A_0: \beta(v)/F_0(v)\in I \}$ with the
convention $0_p\in A_p(\subset A)$ if and only if
$T_pM\setminus\{0_p\}\subset A_p$, for each $p\in M$.  Put
$\psi(s):=\phi^2(s)$ for all $s\in I$ and define the functions:
\be\label{psipsi} \varphi_1:=2\psi-s\dot\psi \; (=2\phi(\phi
-s\dot \phi)), \quad \varphi_2:=2\psi\ddot\psi-\dot\psi^2 \;
(=4\ddot \phi \phi^3).\ee The fundamental tensor $g$ of $F$ is
determined by
\begin{multline}\label{alphabetatensor}
 2  g_v(w,w)= \varphi_1\left(\frac{\beta(v)}{F_0(v)}\right) h^0_v(w,w) \\
+\frac12 \psi\left(\frac{\beta(v)}{F_0(v)}\right)^{-1}\varphi_2\left(\frac{\beta(v)}{F_0(v)}\right)\left(\frac{\beta(v)}{F_0(v)^2}g_v^0(v,w)-\beta(w)\right)^2 \\
+  \frac12 \psi\left(\frac{\beta(v)}{F_0(v)}\right)^{-1}
\left(\varphi_1\left(\frac{\beta(v)}{F_0(v)}\right)\frac{g_v^0(v,w)}{F_0(v)}+\dot\psi\left(\frac{\beta(v)}{F_0(v)}\right)\beta(w)\right)^2,
\end{multline}
where $h_v^0$ is the $F_0$-angular metric defined in
\eqref{angularmetric},  $v\in A \setminus 0$ and $w\in
T_{\pi(v)}M$.

Moreover, $F$ is a conic Finsler metric on $A$ if
\begin{align}
 \varphi_1=2\psi-s\dot\psi>0 \; (\text{i.e.} \, \phi-s\dot\phi>0)
&&\text{and}&& \varphi_2=2\psi(s)\ddot\psi(s)-\dot\psi(s)^2 \geq
0 \; (\text{i.e.} \, \ddot\phi&\geq 0). \label{eqpsi}
\end{align}
\end{prop}
\begin{proof}
To compute $g$, put $L(a,b)=a^2\psi(b/a)$ defined in
\[\{(a,b)\in
\R^2\setminus \{(0,s):s\in\R\}: b/a\in I\}.\] Then,
\begin{align*}
 L_{,1}(a,b)&=2a\, \psi\left(b/a\right)-b\,\dot\psi\left(b/a\right),&\\
L_{,11}(a,b)&=2\, \psi\left(b/a\right)-2\frac{b}{a}\,\dot\psi\left(b/a\right)+\frac{b^2}{a^2}\,\ddot\psi\left(b/a\right),&\\
L_{,12}(a,b)&=\dot\psi\left(b/a\right)-\frac{b}{a}\,\ddot\psi\left(b/a\right),&\\
L_{,22}(a,b)&=\ddot\psi\left(b/a\right).
\end{align*}
The expression of $g_v$ follows easily from
\eqref{fundamentalTensor} and the expressions for $L_{,1},L_{,rs}$
above, recalling that
\begin{multline}\label{hessianoL}
\left(\begin{array}{cc} x_1 & x_2
\end{array}\right){\rm Hess (L)}
\left(\begin{array}{c}
x_1\\
x_2
\end{array}\right)=\frac{2\psi(s)\ddot\psi(s)-\dot\psi(s)^2}{2\psi(s)}(s x_1-x_2)^2\\+\frac{1}{2\psi(s)}((2\psi(s)-s\dot\psi(s))x_1+\dot\psi(s) x_2)^2,
\end{multline}
where  $s=b/a$ and $x_1,x_2\in\R$. The positive definiteness of $g_v$  under conditions \eqref{eqpsi}
is immediate from expression \eqref{alphabetatensor}.


\end{proof}

\begin{rem} \label{rL1defposit} In the set of sufficient conditions $\varphi_1>0,
\varphi\geq 0$  (Eq. \eqref{eqpsi}), the first one is also
necessary (to obtain the positive definiteness of $g$) when $N>2$,
as $L_{,1}(F_0(v),\beta(v)) = F_0(v) \varphi_1(s)$ for
$s=\beta(v)/F(v)$ (recall Proposition \ref{reciproco}).
\end{rem}

\subsubsection{Kropina type metrics}
 Kropina metrics are
$(\alpha, \beta)$-metrics of the form $\alpha/\beta$,  which were
first studied by Kropina in \cite{Kro59} (see also
\cite{Mat91a,Shi78,SiSi85,YoOk07}). In our extension, we consider
not only an $(F_0,\beta)$-metric but also introduce an arbitrary
exponent $q>0$. We remark that in the paper \cite{SPK03}, the
authors study when two Finsler metrics $F$ and $F_0$ are related
by $F=F_0/\beta$ for some one-form $\beta$ and in \cite[formula
(6)]{Bogos}, the authors consider a quadratic Finsler metric
of this type. 

\begin{cor}\label{cKropina} Consider  $F=F_0^{q+1}/|\beta|^{q}$
and $A=\{v\in A_0 :  \beta(v)\not=0\}.$
 Then $F$ is a conic Finsler metric
defined on $A$ for every $q>0$, with fundamental tensor determined
 by
\begin{multline}\label{gkropina}
\left|\frac{\beta(v)}{F_0(v)}\right|^{2q} g_v(w,w)=(q+1)h^0_v(w,w) \nonumber\\
 +q(q+1)
 \left(\frac{g_v^0(v,w)}{F_0(v)}-\frac{F_0(v)}{\beta(v)}\beta(w)\right)^2\\
 +  \left((q+1)\frac{g_v^0(v,w)}{F_0(v)}-q\frac{F_0(v)}{\beta(v)}\beta(w)\right)^2\nonumber
\end{multline}
where $v\in A$ and $w\in T_{\pi(v)}M$.  In particular, any Kropina
metric $F=\alpha/|\beta|$ is strongly convex in its natural domain
of definition.
\end{cor}
\begin{proof}
Observe that $F=F_0^{q+1}/|\beta|^{q}=F_0
\left|F_0/\beta\right|^{q}$. Now if $\phi(s)=|s|^{-q}$,
$s\in\R\setminus \{0\}$, then $F=F_0\phi\left(\beta/F_0\right)$.
Moreover,  inequalities
\eqref{eqpsi} become
\begin{align*}
\varphi_1(s)=2(q+1)|s|^{-2q}&>0,&
\varphi_2(s)=4q(q+1)|s|^{-4q-2}&\geq 0,
\end{align*}
which hold whenever $q>0$.  Therefore, the formula for $g_v$ 
follows from \eqref{alphabetatensor} (use
$|s|^{2q}\psi^{-1}(s)=|s|^{4q}; \dot \psi(s)=-2q|s|^{-2q}
s^{-1}$), and
 positive definiteness is obvious from the expression of $g_v$.
\end{proof}

\subsubsection{Matsumoto type metrics}
Matsumoto metrics were found by this author when measuring the
walking time when there is a slope \cite{Mat89b}. They are
$(\alpha,\beta)$-metrics of the form $\alpha/(\alpha-\beta)$. We extend
them in the same way as Kropina's.

\begin{cor}\label{GenMatsumoto}
For any $q\neq 0$, consider the conic pseudo-Finsler metric:
$$F=F_0^{q+1}/|F_0-\beta|^q, \quad  A=\{v\in A_0 :
F_0(v)\neq \beta(v)\}.  $$ Then, its
 fundamental tensor is determined by
\begin{multline}
 \left|\frac{F_0(v)-\beta(v)}{F_0(v)}\right|^{2q+2}  g_v(w,w)=\\
\frac{(F_0(v)-\beta(v))(F_0(v)-(q+1)\beta(v))}{F_0(v)^2}h^0_v(w,w) \nonumber\\
+q(q+1)\left(\frac{\beta(v)}{F_0(v)^2}g_v^0(v,w)-\beta(w)\right)^2\nonumber\\
 +\left(\frac{F_0(v)-(q+1)\beta(v)}{F_0(v)^2}g_v^0(v,w)+q\beta(w)\right)^2,
\end{multline}
for any $v\in A$ and $w\in T_{\pi(v)}M$.

 When $q>0$ or $q\leq-1$ the restriction of $F$ to
  $$A^*=\{v\in A_0 : (F_0(v)-(q+1)\beta(v))(F_0(v)-\beta(v))>0\}$$
is conic Finsler and,  if $N>2$,   $g$ is not strongly convex at
any point of $A\setminus
 A^*$.
\end{cor}
\begin{proof}
As $F=F_0^{q+1}/|F_0-\beta|^{q}=F_0
\left(F_0/|F_0-\beta|\right)^{q}$, putting $\phi(s)=|1-s|^{-q}$,
$s\not=1$, then $F=F_0 \phi \left(\beta/F_0\right)$. So,
inequalities \eqref{eqpsi} read
\begin{align*}
 \varphi_1(s)=  2 |1-s|^{-2q-2}(1-s)(1-(1+q)s)&>0,\\
\varphi_2(s)= 4 q(q+1) |1-s|^{-4q-2}&\geq 0, \end{align*}
which hold if and only if $(1-s)(1-(1+q)s)>0$ and  either
$q\geq 0$ or $q\leq -1$.   So, the formula for $g_v$ 
follows from \eqref{alphabetatensor} (use
$\psi(s)^{-1}\varphi_2(s) =4q(q+1)/|1-s|^{2q+2}; \dot
 \psi(s)=-2q|1-s|^{-2q} \, (1-s)^{-1}$), and
 strong convexity in $A^*$ from the expression of $g_v$.
 For the last assertion, apply Remark \ref{rL1defposit}
 noticing that $A^*$ determines the region where $L_{,1}>0$.
\end{proof}
\begin{rem}
(1) In the case $N=2$,  the maximal domain $A^*$ where such a
generalized Matsumoto metric is a conic Finsler one, becomes more
involved (see Remark \ref{rL1defposit2}). As a particular case of
Corollary \ref{characterization} in the Appendix, the following
necessary and sufficient condition to define $A^*$ is obtained
(for simplicity, we put  $r=1$):
\[3\beta(v)<2\|\beta \|^2_{g_v^0} +1 \quad \quad \hbox{for any} \, v\in A \; \hbox{with} \; F_0(v)=1. \]

(2)  The class of metrics of the last corollary with $F_0=\alpha$
and $r<-1$, were considered in \cite{Zhou10} in order to obtain
Finsler metrics with constant flag curvature.
\end{rem}
In particular, a known result on Matsumoto metric is recovered.
\begin{cor}\label{matsustrongly} Any
 Matsumoto metric $F=\alpha^2/|\alpha-\beta|$ is  strongly convex  in
\[A^* =\{v\in A_0 :  (\alpha(v)-2\beta(v))(\alpha(v)-\beta(v))>0\}.\]
\end{cor}

 As a last consequence, we consider a class of metrics which
include those  in the references \cite{CuSh09,LiSh07,ShCi08}.

\begin{cor}Let us define $F=(F_0+\beta)^2/F_0$ and
\[A=\{v\in A_0 : F_0(v)^2>\beta(v)^2\}.\]
Then $F$ is strongly convex in $A$.
\end{cor}
\begin{proof}
Apply Corollary \ref{GenMatsumoto} with $q=-2$.
\end{proof}

\subsubsection{Randers type metrics}  Randers metrics are $(\alpha,\beta)$-metrics defined
by $\alpha+\beta$ (they are named after the work of G. Randers
\cite{Ran41} about electromagnetism). Next we will consider
Finsler metrics of the form $F_0+\beta$, which generalize Randers
construction. In particular our result can be used to prove strong
convexity of Randers metric in a more direct way than in
\cite[pag. 284]{BaChSh00}. Moreover, it follows that the
deformations of Kropina metrics considered in \cite{Rom06} are
also strongly convex.
\begin{cor}\label{randers}
 Any  metric $F=F_0+\beta$, where $A=\{v\in A_0 : F_0(v)+
\beta(v)>0\}$, is a conic Finsler metric on all $A$ with
fundamental tensor
\[g_v(w,w)=\frac{F_0(v)+\beta(v)}{F_0(v)}h^0_v(w,w)+
\left(\frac{g_v^0(v,w)}{F_0(v)}+\beta(w)\right)^2,\]
for any $v\in A\setminus 0$ and $w\in T_{\pi(v)}M$.  
\end{cor}
\begin{proof}
 Put $F=F_0\left(1+\frac{\beta}{F_0}\right)$ and define
$\phi(s)=1+s$. As $\varphi_2\equiv 0$, inequalities \eqref{eqpsi}
reduce to $\varphi_1=2(1+s)>0$, and the proof follows from
Proposition \ref{functionpsi}.
 \end{proof}  Notice that, essentially, this result can be
also  regarded as a particular case of Corollary
\ref{rrandersr}.

\subsubsection{A further generalization: $(F_1,F_2)$-metrics}
From now on, $F_1$ and $F_2$ will be two conic Finsler metrics
defined in a common conic domain $A_0$. In order to show the power
of our computations, let us explore the consequences of changing
$\beta$  into a second Finsler metric in $(F_0,\beta)$-metrics.
\begin{prop}\label{functionpsi2}
Let $\phi:I\subset\R\rightarrow \R$, $\phi>0$ be a smooth
function, and put $\psi (= \phi^2)$, $\varphi_1$, $ \varphi_2$
as in Proposition \ref{functionpsi}. Define the conic
pseudo-Finsler metric:
\[F(v)=F_1(v)\phi\left(\frac{F_2(v)}{F_1(v)}\right)\]
on $A=\{v\in A_0: F_2(v)/F_1(v)\in I\}$, with the convention
$0_p\in A_p(\subset A)$ if and only if
$T_pM\setminus\{0_p\}\subset A_p$, for each $p\in M$. The
fundamental tensor $g$ of $F$ is:
\begin{multline}\label{alphabetatensor2}
 2 g_v(w,w)=\varphi_1\left(\frac{F_2(v)}{F_1(v)}\right)
h^1_v(w,w)  +
 \frac{F_1(v)}{F_2(v)} \dot\psi\left(\frac{F_2(v)}{F_1(v)}\right)h^2_v(w,w)\\
+\frac
12\psi\left(\frac{F_2(v)}{F_1(v)}\right)^{-1}\varphi_2\left(\frac{F_2(v)}{F_1(v)}\right)
\left(\frac{F_2(v)}{F_1(v)^2}g^1_v(v,w)-\frac{g^2_v(v,w)}{F_2(v)}\right)^2\\
+\frac 12\psi\left(\frac{F_2(v)}{F_1(v)}\right)^{-1}
\left(\varphi_1\left(\frac{F_2(v)}{F_1(v)}\right)\frac{g^1_v(v,w)}{F_1(v)}+\dot\psi\left(\frac{F_2(v)}{F_1(v)}\right)\frac{g^2_v(v,w)}{F_2(v)}\right)^2,
\end{multline}
for any $v\in A \setminus 0$ and $w\in T_{\pi(v)}M$.
 Moreover, $F$ is a conic Finsler metric on $A$, if $\dot \psi \geq 0$ and Eq.
 \eqref{eqpsi} is satisfied (i.e. $\varphi_1>0, \varphi_2\geq 0$).

\end{prop}
\begin{proof}
The proof is analogous to that of Proposition
\ref{functionpsi}. Just observe that $L_{,2}(a,b)=a\dot\psi(b/a)$
and apply Theorem \ref{central}  as in Proposition
\ref{functionpsi}.
\end{proof}

\begin{cor}\label{GenMatsumoto2}
 For any $q\neq 0$, consider the conic pseudo-Finsler metric:
$$F=F_1^{q+1}/|F_1-F_2|^q, \quad A=\{v\in A_0 :
F_1(v)\neq F_2(v)\}.$$
 Then, its
 fundamental tensor is given by

\begin{multline}
\left(\frac{F_1(v)-F_2(v)}{F_1(v)}\right)^{2q+2}g_v(w,w)=\\
\frac{(F_1(v)-F_2(v))(F_1(v)-(q+1)F_2(v))}{F_1(v)^2}h^1_v(w,w)\nonumber\\
+q \frac{F_1(v)(F_1(v)-F_2(v))}{F_2(v)^2}  h^2_v(w,w)\\
+q(q+1)\left(\frac{F_2(v)}{F_1(v)^2}g^1_v(v,w)-\frac{g^2_v(v,w)}{F_2(v)}\right)^2\nonumber\\
 +\left(\frac{F_1(v)-(q+1)F_2(v)}{F_1(v)^2}g^1_v(v,w)+q\frac{g^2_v(v,w)}{F_2(v)}\right)^2,
\end{multline}
where $v\in A$ and $w\in
T_{\pi(v)}M$.

 When $q>0$ 
 the restriction of $F$ to
$$A^*=\{v\in A_0 : F_1(v)-(q+1)F_2(v)>0 \, \hbox{and} \, F_1(v)-F_2(v)>0\}.$$
is conic Finsler. 
\end{cor}
\begin{proof}
Apply Proposition \ref{functionpsi2} following the same lines as
in Corollary \ref{GenMatsumoto}. Notice that we have to ensure now
$F_1>F_2$ in the definition of $A^*$ because, otherwise,
$\dot\psi(s)$ is not positive (recall the third line in the
expression of $g_v$).
\end{proof}

\subsection{Appendix}\label{sA}
Let us remark that in \cite[Lemma 1.1.2]{ChSh05}, the authors
obtained a result closely related to our Proposition
\ref{functionpsi} for $(\alpha,\beta)$-metrics, concretely:

\begin{crit}[Chern and Shen, \cite{ChSh05}]\label{shenlemma}
$F=\alpha \phi(\beta/\alpha)$  is a Minkowski
norm for any Riemannian metric
$\alpha$ and 1-form $\beta$ with $\|\beta\|_\alpha<b_0$ if and
only if $\phi=\phi(s)>0$ satisfies the following conditions:
\begin{equation}
(\phi(s)-s\dot\phi(s))+(b^2-s^2)\ddot\phi(s)>0,
\label{crit}\end{equation} where $s$ and $b$ are arbitrary numbers
with $|s|\leq b<b_0$, and $\phi$ is defined in a symmetric interval
$(-b_0,b_0)$.
\end{crit}

\begin{rem}\label{rcrit} To compare with  Proposition \ref{functionpsi}, recall
that there,  the hypotheses \eqref{eqpsi} on $\phi$ were
$\phi(s)-s\dot\phi(s)>0$ and $\ddot \phi(s)\geq 0$. Clearly, these
two hypotheses imply \eqref{crit}. Conversely, \eqref{crit}
implies the first of the two hypotheses (choose $b=s$, and notice
that any $b\leq \|\beta\|_\alpha$ can be chosen), but it is
somewhat weaker than the second one. The reason is that the
criterion above assumes
$\|\beta\|_\alpha<b_0$ but no assumption on $\beta$ was done in
Proposition \ref{functionpsi}. However, even in this case, the
criterion can be applied to give our sufficient condition. Namely,
taking an increasing sequence of compact neighborhoods $\{K_j\}$
which exhausts $M$, for each $K_j$ there exists a constant $b_j$
which plays the role of $b_0$ and, if $\{b_j\}\rightarrow \infty$
then condition \eqref{crit} splits into the two conditions of
Proposition \ref{functionpsi}.

For the sake of completeness, Criterion \ref{shenlemma} will be
reobtained next for $(F_0,\beta)$-metric, and its hypothseses will
be stated  in a more accurate way (see Remark \ref{rult}). We will
follow the approach in previous references on the topic computing
the determinant of the matrix $(g_v)_{ij}$ explicitly.
\end{rem}

Let us fix a
coordinate system $\varphi:U\rightarrow \bar{U}\subset \R^N$ and
denote $\frac{\partial}{\partial
x^1},\ldots,\frac{\partial}{\partial x^N}$, the vector fields
associated to the system. Moreover, $g_{ij}(v)$ and $g^0_{ij}(v)$
will denote respectively the coordinates of the fundamental
tensors $g_v$ and $g_v^0$ associated to $F=\sqrt{L(F_0,\beta)}$
and $F_0$ respectively. Then from \eqref{fundamentalTensor} it
follows that
\begin{equation}\label{FTcoordinates}
2g_{ij}(v)=\frac{L_{,1}}{F_0}\left(g^0_{ij}(v)-\frac{1}{F_0(v)^2}v_iv_j\right)+\frac{L_{,11}}{F_0(v)^2}v_iv_j+\frac{L_{,12}}{F_0(v)}(v_ib_j+v_jb_i)+L_{,22}b_ib_j,
\end{equation}
where $v=\sum_{i=1}^N v^i\frac{\partial}{\partial x^i}$, $v_i=\sum_{j=1}^N
g^0_{ij}(v)v^j$ and $b_i=\omega(\frac{\partial}{\partial x^i})$, $i=1,\ldots,N$.
Our goal is to compute the determinant of the matrix
$(g_{ij}(v))$. We will need the following algebraic result whose
proof can be found in \cite[Proposition 30.1]{Mat86} (for
symmetric matrices) or in \cite[Proposition 11.2.1]{BaChSh00}.
\begin{lemma}\label{algebraiclemma}
Let $P=(p_{ij})$ and $Q=(q_{ij})$ be $n\times n$ real matrices and
$C=(c_i)$ be an $n$-vector. Assume that $Q$ is invertible with
$Q^{-1}=(q^{ij})$, and $p_{ij}=q_{ij}+\delta c_ic_j.$ Then
\[\det(p_{ij})=(1+\delta c^2)\det (q_{ij}),\]
where $c:=\sqrt{\sum_{i,j=1}^nq^{ij}c_ic_j}$. Therefore, if $1+\delta
c^2\not=0$, then $P$ is invertible,  and its  inverse matrix
$P^{-1}=(p^{ij})$ is given by
\[p^{ij}=q^{ij}-\frac{\delta c^ic^j}{1+\delta c^2},\]
where $c^i:=\sum_{j=1}^n q^{ij}c_j$.
\end{lemma}
We remark that in \cite{SaShi,SaShiErr} similar computations to
those of the following lemma have been carried out for
$(\alpha,\beta)$-metrics.
\begin{lemma}\label{determinant}
Denote $\Delta_L=\det( {\rm Hess}(L))$. Then
\begin{multline}\label{detgij}
\det
(g_{ij}(v))=\frac{L_{,1}^{N-2}}{2^NF_0(v)^{N+1}}\left(F_0(v)\Delta_L(F_0(v)^2
\|\beta\|_{g_v^0}^2-\beta(v)^2) +2L_{,1} L\right) \det
(g^0_{ij}(v)).
\end{multline}
\end{lemma}
\begin{proof}
First observe that
\[g_{ij}(v)=\frac{L_{,1}}{2F_0(v)}(g^0_{ij}(v)+\delta(v) (b_i+k(v) v_i)(b_j+k(v) v_j)+\mu(v) v_iv_j),\]
where
\begin{align*}
&\delta(v)=\frac{L_{,22} F_0(v)}{L_{,1}},&
k(v)=\frac{L_{,12}}{L_{,22}F_0(v)}&& \text{and}
&&\mu(v)=\frac{\Delta_L}{L_{,1}L_{,22}F_0(v)}-\frac{1}{F_0(v)^2}.
\end{align*}
If we call $c_i=b_i+k(v) v_i$ and
$g^0(v)^{ij}=(g^0_{ij})^{-1}(v)$, applying twice Lemma
\ref{algebraiclemma}, we obtain that
\begin{multline*}
\det(g_{ij}(v))=\frac{
L_{,1}^N}{2^NF_0(v)^N}((1+\mu(v)\sum_{i,j= 1}^N
g^0(v)^{ij}v_iv_j)(1+\delta(v)
c^2)\\-\mu(v)\delta(v)\sum_{i,j=1}^Nc^ic^jv_iv_j) \det (g^0_{ij}(v)),
\end{multline*}
where
\begin{multline*}
 c^2=\sum_{i,j=1}^N g^0(v)^{ij}c_ic_j=\sum_{i,j=1}^Ng^0(v)^{ij}(b_i+k(v) v_i)(b_j+k(v) v_j)\\=\|\beta\|^2_{g_v^0}+2k(v) \beta(v)+k(v)^2 F_0(v)^2,
 \end{multline*}
$\sum_{i,j= 1}^N g^0(v)^{ij} v_iv_j=F_0(v)^2$ and
\[\sum_{i,j=1}^N c^ic^j v_iv_j=\sum_{i,j=1}^N(b^i+k(v) v^i)(b^j+k(v) v^j)v_iv_j=
(\beta(v)+k(v) F_0(v)^2)^2.\]
Then, formula \eqref{detgij}
follows
 substituting $\delta(v)$, $k(v)$ and $\mu(v)$ by their values,
and making some straightforward computations (use
\[F_0(v)^2 L_{,11}+\beta(v)^2L_{,22}+2\beta(v)F_0(v)L_{,12}=2L,\]
which follows from evaluating \eqref{fundamentalTensor} in $w=v$ and
recalling \eqref{propfundtensor}).
\end{proof}
\begin{prop}\label{pdetg}
If $F=F_0\phi(\beta/F_0)$, then $\det(g_{ij}(v))$ is equal to
\begin{multline}\label{detphi}
\left(\phi-\frac{\beta(v)}{F_0(v)}\dot\phi\right)^{N-2}\left(\left(\|\beta\|_{g_v^0}^2-\frac{\beta(v)^2}{F_0(v)^2}\right)\ddot\phi 
+\phi-\frac{\beta(v)}{F_0(v)}\dot\phi\right)\phi^{N+1} \det
(g^0_{ij}(v)).
\end{multline}
\end{prop}
\begin{proof}
 Using the expressions of the partial derivatives of $L$ computed
in Proposition \ref{functionpsi}, Equation \eqref{detgij} becomes,
 in terms of $\psi=\phi^2$,
\begin{multline*}
\det (g_{ij}(v))=\frac{(2F_0(v)\psi-\beta
\dot\psi)^2}{2^NF_0(v)^{N+1}}((2\psi\ddot\psi-(\dot\psi)^2)F_0(v)(F_0(v)^2\|\beta\|^2_{g_v^0}-\beta(v)^2)\\+(2F_0(v)\psi-\beta(v)\dot\psi)2\psi
F_0(v)^2).
\end{multline*}
Then substituting $\psi=\phi^2$, $\dot\psi=2\phi\dot\phi$ and
$\ddot\psi=2((\dot\phi)^2+\phi\ddot\phi)$ and making some
straightforward computations we obtain \eqref{detphi}.
\end{proof}
\begin{cor}\label{characterization}
Let $F=F_0\phi(\beta/F_0)$, choose $v_0\in A\setminus 0$ (see
Proposition \ref{functionpsi}), and put $s_0 =\beta(v_0)/F_0(v_0)$.
In the case of dimension $N>2$, the fundamental tensor $g_{v_0}$
is positive definite if and only if
\begin{align}\label{etomashen}
\phi(s_0)-s_0\dot\phi(s_0)>0,\\
\ddot\phi(s_0)\left(\|\beta\|_{g_{v_0}^0}^2-s_0^2\right) +\phi(s_0)-s_0
\dot\phi(s_0)>0
\end{align}
and, in the case of dimension $N=2$, $g_{v_0}$ is positive
definite if and only if the second  inequality holds.
\end{cor}
\begin{proof}
Define $\phi_t(s)=1-t+t\phi(s)$, in the same domain as $\phi$ for
$t\in[0,1]$. Then, in dimension $N>2$
\begin{align*}
\phi_t(s_0)-s_0\dot\phi_t(s_0)=1-t+t(\phi(s_0)-s_0\dot\phi(s_0))>0,
\end{align*}
and in any dimension $N\geq 2$
\begin{multline*}
\ddot\phi_t(s_0)\left(\|\beta\|_{g_{v_0}^0}^2-s_0^2\right)
+\phi_t(s_0)-s_0\dot\phi_t(s_0)\\=1-t+t\left[\ddot\phi(s_0)\left(\|\beta\|_{g_{v_0}^0}^2-s_0^2\right)
+\phi(s_0)-s_0\dot\phi(s_0)\right]>0.
\end{multline*}
Then, let $F_t(v_0)=F_0(v_0)\phi_t(\beta(v_0)/F_0(v_0))$,
$g_{v_0}^t$,  the fundamental tensor associated to $F_t$, and
$(g_{ij}^t(v_0))$  the matrix of coordinates of $g_{v_0}^t$ in the
 coordinate chart  previously fixed.  Applying
Proposition \ref{pdetg} to each $\phi_t$, one has
$\det(g_{ij}^t(v_0))>0$   for every $t\in [0,1]$. Observe that
for $t=0$, $F_t=F_0$, and, consequently, $(g_{ij}^0(v_0))$ is
positive definite. Then, as $\det(g_{ij}^t(v_0))>0$ for every
$t\in [0,1]$, $g_{ij}^1(v_0)$
must be positive definite. 

For the converse, observe that Proposition \ref{reciproco} implies
that $\phi(s_0)-s_0\dot\phi(s_0)>0$ for $N>2$, and the other inequality
follows for any $N\geq 2$ by using  \eqref{detphi}, as
$\det(g_{ij}(v_0))$ must be positive.
\end{proof}

\begin{rem}\label{rult}
The inequalities in the previous corollary characterize in a
precise way the maximal conic domain $A^*$ where  $F$ is conic
Finsler. In the second inequality, observe that
$\|\beta\|_{g_{v_0}^0}^2-s_0^2\geq 0$ (this  follows directly from
the definition of the norm $\|\beta\|_{g_{v_0}^0}$ and the
equality $F_0(v_0)=\sqrt{g_{v_0}^0(v_0,v_0)}$).

In the case that $F_0=\sqrt{\alpha}$ (i.e. $F$ is a conic
$(\alpha, \beta)$-metric) then $\|\beta\|_{g_{v_0}^0}$ depends
only on the point $p_0=\pi(v_0)$ and we write
$\|\beta\|_{\alpha_{p_0}}$ instead. Moreover, if $F$ is an
$(\alpha, \beta)$-metric (with $A=TM$), then an $\alpha$-unit
vector $v_0$ can be chosen both, in the kernel of $\beta$ and such
that $\|\beta\|_{\alpha_{p_0}} = \beta(v_0)$, and as a
consequence,   $s^2=\beta(v)^2/\alpha(v)^2$, $v\in
A_{p_0}\setminus \{0\}$, takes all the values in the interval
$[0,\|\beta\|_{\alpha_{p_0}}^2]$ (notice that, in this case, $0$
must belong to the domain $I$ of definition of $\phi$). Thus, the
following criterion (to be compared with Criterion
\ref{shenlemma}, taking into account Remark \ref{rcrit})  is
obtained: {\em an $(\alpha,\beta$)-metric $F=\alpha
\phi(\beta/\alpha)$ is a Finsler one if and only if the
inequalities in Corollary \ref{characterization} hold for all
$s\in \R$ such that $0\leq s^2 \leq \|\beta\|_\alpha:=$ {\em
Sup}$_{p\in M}\|\beta\|_{\alpha_{p}}$.}, i.e.:
\begin{align}\label{etomashen2}
\phi(s)-s\dot\phi(s)>0,\\
\ddot\phi(s)\left(\|\beta\|_{\alpha}^2-s^2\right) +\phi(s)-s
\dot\phi(s)>0,
\end{align}
the latter under the convention that, if $\|\beta\|_\alpha=\infty$
then $\ddot \phi\geq 0$ on $I$.

\end{rem}

\section*{Acknowledgments}

The authors warmly acknowledge  Professors R. Bryant, E. Caponio
and G. Siciliano for helpful conversations on the topics of this
paper and the anonimous referee for his interesting comments.

Both authors are partially supported by Regional J. Andaluc\'{\i}a
Grant P09-FQM-4496 with FEDER funds. MAJ  is also partially
supported by MICINN project MTM2009-10418 and Fundaci\'on S\'eneca
project 04540/GERM/06 and MS by MICINN-FEDER Grant MTM2010--18099.

\end{document}